\documentclass[10pt]{amsart}
\usepackage[utf8]{inputenc}
\usepackage[T1]{fontenc}
\usepackage[english]{babel}
\title[Local limit theorems in relatively hyperbolic groups I]{Local limit theorems in relatively hyperbolic groups I : rough estimates}
\author{Matthieu Dussaule}
\date{}

\usepackage{comment}
\usepackage{amsmath}
\usepackage{amssymb}
\usepackage{dsfont}
\usepackage{amsfonts}
\usepackage{amsthm}
\usepackage{tikz}
\usepackage{graphicx}
\usepackage{fancyhdr}
\usepackage{caption}
\usepackage{enumerate}
\usepackage{mathtools}
\usepackage[colorlinks=true,pdfstartview=FitH,bookmarks=false]{hyperref}
\hypersetup{urlcolor=black,linkcolor=black,citecolor=black,colorlinks=true}

\newcommand{\vertiii}[1]{{\left\vert\kern-0.25ex\left\vert\kern-0.25ex\left\vert #1 
    \right\vert\kern-0.25ex\right\vert\kern-0.25ex\right\vert}}

\newcommand\Z{\mathbb{Z}}

\newcommand\R{\mathbb{R}}
\newcommand\C{\mathbb{C}}

\theoremstyle{plain}
\newtheorem{definition}{Definition}[section]
\newtheorem{proposition}[definition]{Proposition}
\newtheorem{corollary}[definition]{Corollary}
\newtheorem{theorem}[definition]{Theorem}
\newtheorem{lemma}[definition]{Lemma}
\newtheorem*{prop*}{Proposition}
\newtheorem*{lem*}{Lemma}

\theoremstyle{remark}
\newtheorem{remark}{Remark}[section]

\newtheorem*{rem*}{Remark}

\DeclareMathOperator{\Cay}{Cay}

\usepackage{etoolbox}
\apptocmd{\sloppy}{\hbadness 10000\relax}{}{}
\apptocmd{\sloppy}{\vbadness 10000\relax}{}{}

\begin{document}

\begin{abstract}
    This is the first of a series of two papers dealing with local limit theorems in relatively hyperbolic groups.
    In this first paper, we prove rough estimates for the Green function.
    Along the way, we introduce the notion of relative automaticity which will be useful in both papers and we show that relatively hyperbolic groups are relatively automatic.
    We also define the notion of spectral positive-recurrence for random walks on relatively hyperbolic groups.
    We then use our estimates for the Green function to prove that $p_n\asymp R^{-n}n^{-3/2}$ for spectrally positive-recurrent random walks, where $p_n$ is the probability of going back to the origin at time $n$ and where $R$ is the spectral radius of the random walk.
\end{abstract}

\maketitle

\section{Introduction}

\subsection{Random walks and local limit theorems}
Consider a finitely generated group $\Gamma$ and a probability measure $\mu$ on $\Gamma$.
We define the $\mu$-random walk on $\Gamma$, starting at $\gamma \in \Gamma$, as
$X_n^{\gamma}=\gamma g_1...g_n$,
where $(g_k)$ are independent random variables of law $\mu$ in $\Gamma$.
We say that $\mu$ is admissible if its support generates $\Gamma$ as a semigroup.
Equivalently, one can reach any point from any point with the random walk with positive probability.
We will always assume in the following that measures $\mu$ are admissible.
We say that $\mu$ is symmetric if $\mu(\gamma)=\mu(\gamma^{-1})$,
which means for the random walk that the probability to go from $\gamma$ to $\gamma'$ is the same as the probability to go from $\gamma'$ to $\gamma$.
The law of $X_n^{\gamma}$ is denoted by $p_n(\gamma,\gamma')=p_n(e,\gamma^{-1}\gamma')$.
It is given by the convolution powers of $\mu$, that is, $p_n(e,\gamma)=\mu^{*n}(\gamma)$.

Say that the $\mu$-random walk is aperiodic if $p_n(e,e)>0$ for large enough $n$.
The local limit problem consists in finding asymptotics of $p_n(e,e)$ when $n$ goes to infinity.
In many situations, if the $\mu$-random walk is aperiodic, one can prove a local limit theorem of the form
\begin{equation}\label{locallimittheoremgeneralform}
p_n(e,e)\sim C R^{-n}n^{-\alpha},
\end{equation}
where $C>0$ is a constant, $R\geq 1$ and $\alpha \in \R$.
In such a case, $\alpha$ is called the critical exponent of the random walk.

For example, if $\Gamma=\Z^d$ and $\mu$ is finitely supported and aperiodic, then classical Fourier computations show that
$p_n(e,e)\sim Cn^{-d/2}$ if the random walk is centered and $p_n(e,e)\sim CR^{-n}n^{-d/2}$ with $R>1$ if the random walk is non-centered.
If $\Gamma$ is a non-elementary Gromov-hyperbolic group and $\mu$ is finitely supported, symmetric and aperiodic, then one has
$p_n(e,e)\sim CR^{-n}n^{-3/2}$, for $R>1$.
This was proved by P.~Gerl and W.~Woess \cite{GerlWoess} and S.~Lalley \cite{Lalley} for free groups, by S.~Lalley and S.~Gou\"ezel \cite{GouezelLalley} for cocompact Fuchsian groups and by S.~Gou\"ezel \cite{Gouezel1} for any hyperbolic group.

In \cite{Gerl}, P.~Gerl conjectured that if a local limit of the form~(\ref{locallimittheoremgeneralform}) holds for a finitely supported random walk, then $\alpha$ is a group invariant, that is, if two different finitely supported measures $\mu_1$ and $\mu_2$ lead to asymptotics like~(\ref{locallimittheoremgeneralform}), with $C_1,R_1,\alpha_1$ and $C_2,R_2,\alpha_2$ respectively, then $\alpha_1=\alpha_2$.
D.~Cartwright disproved in \cite{Cartwright2} this conjecture with a spectacular result, constructing different nearest neighbor random walks on a free product $\Z^d*\Z^d$ with different critical exponents, namely $3/2$ and $d/2$, where $d\geq 5$.
He had previously proved that one could have $\alpha=d/2$ for some free products of the form $\Z^d*\Z^d*...*\Z^d$ in \cite{Cartwright1},
whereas W.~Woess had proved in \cite{Woess2} that for nearest neighbor random walks on free products, in general, one had $\alpha=3/2$ (what W.~Woess calls "typical cases" in \cite{Woess}).

\medskip
In \cite{CandelleroGilch}, E.~Candellero and L.~Gilch gave a complete description of every local limit theorem that can occur for finitely supported adapted nearest neighbor random walks on free products of free abelian groups.
Precisely, let $\Gamma=\Z^{d_1}*\Z^{d_2}$ and let $\mu$ be a probability measure on $\Gamma$.
We say that $\mu$ is adapted if it is of the form
$t\mu_1+(1-t)\mu_2$, where $0< t< 1$ and where $\mu_i$ is a probability measure on $\Z^{d_i}$.
We assume that $\mu_i$ is finitely supported and admissible on $\Z^{d_i}$ so that $\mu$ also is finitely supported and admissible on $\Gamma$.
Then, depending on the weight $t$, one can have the critical exponent of $\mu$ to be $d_1/2$, $d_2/2$ or $3/2$.
Every case can occur when $d_1,d_2\geq 5$, see \cite[Section~7]{CandelleroGilch} for more details.

In all the examples cited above, except for the case of $\Z^d$ where explicit computations are made easily,
one shows a local limit theorem by studying the Green function
$$G(\gamma,\gamma')=\sum_{n\geq 0}\mu^{*n}(\gamma^{-1}\gamma')=\sum_{n\geq 0}p_n(\gamma,\gamma'),$$
which encodes the behavior of $p_n(e,e)$.
More generally, define
$$G(\gamma,\gamma'|r)=\sum_{n\geq 0}r^n\mu^{*n}(\gamma^{-1}\gamma').$$
Let $R_{\mu}$ be the radius of convergence of this power series.
We call $R_{\mu}$ the spectral radius of the random walk.
A good way to find asymptotics of $\mu^{*n}(e)$ is to look at singularities of the Green function at the spectral radius.
When the Green function is an algebraic function, this can be done by using Darboux-like theorems.
However, in general, we do not know if it is algebraic and one has to find another way to connect properties of $G(e,e|R_{\mu})$ with asymptotics of $\mu^{*n}(e)$.
In \cite{GouezelLalley} and \cite{Gouezel1}, this is done by using Tauberian theory and this will be our approach here.

\medskip
As we saw, free products provide a great source of examples for local limit theorems.
There are several equivalent ways of defining relatively hyperbolic groups (see Section~\ref{Sectionrelativelyhyperbolicgroups} for more details).
If $\Omega$ is a collection of subgroups of $\Gamma$, we say that $\Gamma$ is hyperbolic relative to $\Omega$ if it acts geometrically finitely on a proper geodesic hyperbolic space $X$ such that the stabilizers of the parabolic limit points are exactly the subgroups in $\Omega$.
The elements of $\Omega$ are called peripheral subgroups or (maximal) parabolic subgroups.
We fix a collection $\Omega_0$ of representatives of conjugacy classes of $\Omega$.
According to \cite[Proposition~6.15]{Bowditch}, such a collection is finite.
If $\Gamma$ is a free product of the form $\Gamma=\Gamma_1*...*\Gamma_n$, then $\Gamma$ is relatively hyperbolic with respect to conjugacy classes of the free factors $\Gamma_i$ and one can choose $\Omega_0=\{\Gamma_1,...,\Gamma_n\}$.

In this series of two papers, we extend W.~Woess' results \cite{Woess2} on free products to any relatively hyperbolic group.
The second paper will be devoted to proving an asymptotic of the form~(\ref{locallimittheoremgeneralform}), with $\alpha=3/2$, for non-spectrally degenerate random walks.
This property of spectral degeneracy was introduced in \cite{DussauleGekhtman} to study stability of the Martin boundary.
In this first paper, we introduce a looser condition, namely spectral positive-recurrence and prove some weaker estimates than~(\ref{locallimittheoremgeneralform}) for spectrally positive-recurrent random walks.

We insist on the fact that both papers are different and use very different techniques.
In particular, the second one is not an enhanced version of the first one.

\medskip
Before stating our main result, we introduce the following terminology.
Let $\mu$ be a finitely supported probability measure on a group $\Gamma$, which is relatively hyperbolic with respect to $\Omega$
and choose a finite collection $\Omega_0$ of representatives of conjugacy classes of $\Omega$.
We can look at the weight given by $\mu$ to a parabolic subgroup $\mathcal{H}$ in $\Omega_0$ in several ways.
One way is to compute the spectral radius of the induced random walk on $\mathcal{H}$ and to see if this spectral radius is 1.
This leads to the notion of spectral degenerescence and we refer to Section~\ref{Sectionspectralpositiverecurrence} for more details.
A weaker way is as follows.
Let $\mathcal{H}\in \Omega_0$.
We define the series of Green moments of $\mu$ along $\mathcal{H}$ as
$$I_\mathcal{H}^{(2)}(r)=\sum_{h,h'\in \mathcal{H}}G(e,h|r)G(h,h'|r)G(h',e|r).$$
\begin{definition}\label{definitionGreenmoments}
We say that $\mu$ (or equivalently the random walk) has finite Green moments if for every $\mathcal{H}\in \Omega_0$,
$I_\mathcal{H}^{(2)}(R_\mu)<+\infty$.
\end{definition}

\begin{definition}\label{definitiondivergent}
We say that $\mu$ (or equivalently the random walk) is divergent if $\frac{d}{dr}_{|r=R_\mu}G(e,e|r)=+\infty$ and that it is convergent otherwise.
\end{definition}

\begin{definition}\label{definitionspectralpositiverecurrent}
We say that $\mu$ (or equivalently the random walk) is spectrally positive-recurrent if it is divergent and has finite Green moments.
\end{definition}

We will give more explanations on these definitions in Section~\ref{Sectionspectralpositiverecurrence}.
Our main result is as follows.
For two functions $f,g$, we write $f\lesssim g$ if there exists $C$ such that $f\leq Cg$.
We write $f\asymp g$ if both $f\lesssim g$ and $g\lesssim f$.
If the implicit constant $C$ depends on some parameters, we will avoid using these notations, except if the dependency is clear from the context.

\begin{theorem}\label{maintheorem}
Let $\Gamma$ be a non-elementary relatively hyperbolic group.
Let $\mu$ be a finitely supported, admissible and symmetric probability measure on $\Gamma$.
Assume that the corresponding random walk is aperiodic and spectrally positive-recurrent.
Then,
$$p_n(e,e)\asymp R_\mu^{-n}n^{-3/2}.$$
\end{theorem}

This theorem is a consequence of the following one.

\begin{theorem}\label{maintheoremGreen}
Let $\Gamma$ be a non-elementary relatively hyperbolic group.
Let $\mu$ be a finitely supported, admissible and symmetric probability measure on $\Gamma$.
Assume that the corresponding random walk has finite Green moments.
Then, as $r$ tends to $R_\mu$,
$$\frac{d^2}{dr^2}G(e,e|r)\asymp \left (\frac{d}{dr}G(e,e|r)\right )^3.$$
\end{theorem}

\subsection{Organization of the paper}\label{Sectionorganization}
We now give some more details about our proofs and briefly describe the content of the paper.

We first give in Section~\ref{Sectionrelativelyhyperbolicgroups} several equivalent definitions of relatively hyperbolic groups and we review basic results about such groups.

In Section~\ref{Sectionpreliminaries}, we give some preliminary results on the Green function that will be used along the paper.
We first give expressions of the derivatives of the Green function in terms of spatial sums on the group, following the work of S.~Gou\"ezel and S.~Lalley, see precisely Lemma~\ref{lemmageneralformuladerivatives}.
As announced in the introduction, we also give more explanations on Definitions~\ref{definitionGreenmoments},~\ref{definitiondivergent} and~\ref{definitionspectralpositiverecurrent}, making an analogy with similar definitions in the context of counting theorems on Kleinian groups.
We recall the definition of spectral degenerescence, introduced in \cite{DussauleGekhtman}, and explain why non-spectral degenerescence implies finiteness of the Green moments.

In Section~\ref{Sectioncoding}, we introduce the notion of relative automaticity and we prove that relatively hyperbolic groups are relatively automatic, see Definition~\ref{definitionautomaticstructure} and Theorem~\ref{thmcodingrelhypgroups}.
The proof is analogous to Cannon's proof \cite{Cannon} for coding geodesics in hyperbolic groups (see also \cite{GhysHarpe}).
We show that there is a finite number of what we call relative cone-types.
This should be compared with the notion of partial cone type introduced in \cite{Yang2},
where it is shown that there is a finite number of them too.
This result is not surprising and is kind of implicit in B.~Farb's work (see \cite{Farbthesis}, \cite{Farb}), although it is not properly stated in there.
We do use the notion of relatively automatic groups in the present paper, but this is mainly for convenience.
On the contrary, it will be of great importance in the next paper.
We will use there relatively automatic structures to relate asymptotic properties of the Green function with asymptotic properties of some operators defined on a countable Markov shift associated with a relatively hyperbolic group.
Theorem~\ref{thmcodingrelhypgroups} will thus be a crucial tool.

In Section~\ref{SectionGreenparabolics}, we prove Theorem~\ref{maintheoremGreen}.
A similar result is given in \cite{Gouezel1} for hyperbolic groups but our proof is more difficult and our estimates also involve the second derivative of the Green function associated with the first return transition kernels on the parabolic subgroups.
Along the way, we obtain  other estimates involving the Green function of the parabolic subgroups and the Green function of the whole group.
We also show that spectral degenerescence of the measure $\mu$ in the sense of \cite{DussauleGekhtman} implies that $\mu$ is divergent in the sense of Definition~\ref{definitiondivergent}, so that non-spectral degenerescence implies spectral positive-recurrence.

We finish the proof of our main result in Section~\ref{SectionTauberian}.
We prove a weak version of Karamata's Tauberian theorem and use it, together with technical results from \cite{GouezelLalley}, to deduce Theorem~\ref{maintheorem} from Theorem~\ref{maintheoremGreen}.

\medskip
Finally, let us mention that we will repeatedly use weak relative Ancona inequalities up to the spectral radius.
These inequalities state the following.
Let $\Gamma$ be a relatively hyperbolic group, let $\mu$ be an admissible probability measure on $\Gamma$.
Let $x,y,z\in \Gamma$ be such that $y$ is on a relative geodesic from $x$ to $z$, or more generally $y$ is within a uniform bounded distance of a transition point on a geodesic from $x$ to $z$ in the Cayley graph of $\Gamma$ (see Section~\ref{Sectionrelativelyhyperbolicgroups} for more details on these notions).
Then, for every $r\leq R_\mu$,
\begin{equation}\label{equationAncona}
    G(x,z|r)\asymp G(x,y|r)G(y,z|r).
\end{equation}
The implicit constant is asked not to depend on $r,x,y,z$.
In other words, the Green function is roughly multiplicative along relative geodesics (uniformly in $r$).
One of the main result of \cite{DussauleGekhtman} is that weak relative Ancona inequalities hold for any finitely supported, symmetric, admissible probability measure $\mu$ on a relatively hyperbolic group $\Gamma$.
%In the appendix, we prove that they are satisfied for finitely supported admissible probability measures on free products, removing the symmetry assumption.
%We use the framework developped by Y.~Derriennic in \cite{Derriennic}, which is itself inspired by the work of G.~Birkhoff in \cite{Birkhoff}.
Note that there also exist strong relative Ancona inequalities,
which were also proved in \cite{DussauleGekhtman}.
%and we will also prove them without any symmetry assumption in the appendix.
Since strong inequalities are technical to state and since we will not need them in this paper, we do not mention them here.
However they will be a crucial tool in the second paper.

\subsection{Acknowledgements}\label{Sectionacknowledgements}
The author thanks I.~Gekhtman and L.~Potyagailo for many helpful conversations about relatively hyperbolic groups
and S.~Tapie for explanations on counting theorems.
He also thanks Mateusz Kwaśnicki for his help on weak Tauberian theorems.

\section{Relatively hyperbolic groups}\label{Sectionrelativelyhyperbolicgroups}
There are several equivalent definitions of relative hyperbolicity.
We saw one above in terms of geometric actions.
Let us give more details now.

Consider a finitely generated group $\Gamma$ acting discretely and by isometries on a proper and geodesic hyperbolic space $(X,d)$.
Choose a base point $o\in X$.
The limit set of $\Gamma$ is the set of accumulation points of the orbit of $o$ in the Gromov boundary of $X$.
It does not depend on the choice of $o$.
We denote this limit set by $\Lambda \Gamma$.
Recall that an element $\gamma$ of $\Gamma$ is called elliptic if it fixes some point in $X$.
Otherwise, either it fixes exactly one point in $\Lambda \Gamma$, or it fixes exactly two points in $\Lambda \Gamma$, one being attractive and the other one being repelling.
In the first case, $\gamma$ is called parabolic and in the second case, it is called loxodromic, see \cite[Chapter~8]{GhysHarpe} for more details.

A point $\xi\in \Lambda \Gamma$ is called conical if there is a sequence $(\gamma_{n})$ of $\Gamma$ and distinct points $\xi_1,\xi_2$ in $\Lambda \Gamma$ such that
$\gamma_{n}\xi$ converges to $\xi_1$ and $\gamma_{n}\zeta$ converges to $\xi_2$ for all $\zeta\neq \xi$ in $\Lambda \Gamma$.
A point $\xi\in \Lambda \Gamma$ is called parabolic if its stabilizer in $\Gamma$ is infinite, fixes exactly $\xi$ in $\Lambda\Gamma$ and contains no loxodromic element.
A parabolic limit point $\xi$ in $\Lambda \Gamma$ is called bounded parabolic if its stabilizer in $\Gamma$ is infinite and acts cocompactly on $\Lambda \Gamma \setminus \{\xi\}$.

Say that the action of $\Gamma$ on $X$ is geometrically finite if the limit set only consists of conical limit points and bounded parabolic limit points.
Then, say that $\Gamma$ is relatively hyperbolic with respect to $\Omega$ if it acts geometrically finitely on such a hyperbolic space $(X,d)$ such that the stabilizers of the parabolic limit points are exactly the elements of $\Omega$.
We say that $\Gamma$ is elementary if the limit set only consists of zero, one or two points.
Otherwise it is infinite.
In the following, we will always assume that the group $\Gamma$ is non-elementary.

Moreover, we can always assume (and we will) that the limit set is the entire Gromov boundary of $X$, up to changing $X$ into the closed convex hull of the limit set, see \cite[Section~6]{Bowditch}.
Denote by $\partial X$ this Gromov boundary.
One might choose different spaces $X$ on which $\Gamma$ can act geometrically finitely.
However, different choices of $X$ give rise to equivariantly homeomorphic boundaries $\partial X$.
We call $\partial X$ the Bowditch boundary of $\Gamma$ and we denote it by $\partial_B\Gamma$ when we de not want to refer to a space $X$ on which $\Gamma$ acts.

\medskip
We will also be interested in a combinatorial description of relatively hyperbolic groups.
This was developed by B.~Farb in \cite{Farbthesis}, see also \cite{Farb}, but it will be more convenient to use D.~Osin's terminology \cite{Osin}.
Consider a finitely generated group $\Gamma$, together with a collection of subgroups $\Omega$ and assume that $\Omega$ is stable by conjugacy and has a finite number of conjugacy classes.
Choose one representative of each conjugacy class to form a set of representatives $\Omega_0=\{\mathcal{H}_1,...,\mathcal{H}_N\}$.

Consider a finite generating set $S$ of $\Gamma$ and denote by $\Cay(\Gamma,S)$ the corresponding Cayley graph.
We will also use the Cayley graph of $\Gamma$ with respect to the infinite generating set $S\cup \bigcup_{1\leq n\leq N}\mathcal{H}_n$.
To avoid confusion in the terminology, we will denote this other Cayley graph by $\hat{\Gamma}$, in reference to Farb's notations.
We denote by $\hat{d}$ the graph distance in $\hat{\Gamma}$.

\begin{definition}
Say that $\Gamma$ is weakly relatively hyperbolic with respect to $\Omega$ if $\hat{\Gamma}$ is hyperbolic.
This is independent of the choices of $\Omega_0$ and $S$.
\end{definition}

In the following, what we will call a path in $\Gamma$ or in $\hat{\Gamma}$ will be a sequence of \textit{adjacent} vertices in the corresponding graph, not just any sequence of vertices.
We label the edges of the path with the corresponding element of $S$ or $S\cup \bigcup \mathcal{H}_n$.
Note that a path in $\Gamma$ induces a path in $\hat{\Gamma}$ but the converse is not true in general.

A relative geodesic is then a path of minimal length between its endpoints in $\hat{\Gamma}$.
A relative $(\lambda,c)$-quasi geodesic path is a path $\alpha=(\gamma_1,...,\gamma_n)$ in $\hat{\Gamma}$ which is also a $(\lambda,c)$-quasi geodesic, that is for all $k,l$,
$$\frac{1}{\lambda}|k-l|-c\leq \hat{d}(\gamma_k,\gamma_l)\leq \lambda|k-l|+c.$$

We say that a path enters the coset $\gamma \mathcal{H}_n$ of a parabolic subgroup $\mathcal{H}_n$ if there is a vertex in this path which is an element of $\gamma \mathcal{H}_n$ and which is followed by an edge labeled with an element of $\mathcal{H}_n$.
Consider then a maximal subpath with vertices in $\gamma \mathcal{H}_n$ and labeled with elements of $\mathcal{H}_n$.
Such a subpath is called a $\mathcal{H}_n$-component.
The entering point (respectively exit point) of the $\mathcal{H}_n$-component is the first (respectively last) vertex of this subpath and we say that the path leaves $\gamma \mathcal{H}_n$ at the exit point.

We also say that the path travels more than $r$ in $\gamma \mathcal{H}_n$ if the distance in $\Cay(\Gamma,S)$ between the entering point and the exit point is larger than $r$.

Finally, we  say that a path is without backtracking if once it has left a coset $\gamma \mathcal{H}_n$, it never goes back to it.

\begin{definition}
Say that the pair $(\Gamma,\Omega)$ satisfies the bounded coset penetration property (BCP for short) if for all $\lambda,c$, there exists a constant $C_{\lambda,c}$ such that for every pair $(\alpha_1,\alpha_2)$ of relative $(\lambda,c)$-quasi geodesic paths without backtracking, starting and ending at the same point in $\Gamma$, the following holds
\begin{enumerate}
\item if $\alpha_1$ travels more than $C_{\lambda,c}$ in a coset, then $\alpha_2$ enters this coset,
\item if $\alpha_1$ and $\alpha_2$ enter the same coset, the two entering points and the two exit points are $C_{\lambda,c}$-close to each other in $\Cay(\Gamma,S)$.
\end{enumerate}
\end{definition}

Again, this definition does not depend on the choices of $\Omega_0$ and $S$.
The following is proved in \cite{Osin}.
\begin{proposition}
The group $\Gamma$ is relatively hyperbolic with respect to $\Omega$ if and only if it is weakly relatively hyperbolic with respect to $\Omega$ and if the pair $(\Gamma,\Omega)$ satisfies the BCP property.
\end{proposition}

Note that mapping class groups are weakly relatively hyperbolic (see \cite{MasurMinsky}) but are not relatively hyperbolic (see \cite{BehrstockDrutuMosher}).
Our finite automaton coding relative geodesics would still code relative geodesics in mapping class groups but would not be finite anymore (we do use the BCP property a lot).
However, one could maybe arrange it to have only finitely many recurrence classes.

\medskip
One important aspect in relatively hyperbolic groups is the notion of transition points.
If $\alpha$ is a geodesic in the Cayley graph $\Cay(\Gamma,S)$, a point on $\alpha$ is called a transition point if it is not deep in a parabolic subgroup.
More precisely, let $\eta_1,\eta_2>0$.
A point $\gamma$ on a geodesic $\alpha$ in $\Cay(\Gamma,S)$ is called $(\eta_1,\eta_2)$-deep if the part of $\alpha$ containing the points at distance at most $\eta_2$ from $\gamma$
is contained in the $\eta_1$-neighborhood of a coset $\gamma_0\mathcal{H}$ of a parabolic subgroup $\mathcal{H}$.
Otherwise, $\gamma$ is called a $(\eta_1,\eta_2)$-transition point.
We refer to \cite{Hruska} and \cite{GerasimovPotyagailo} for more details.

Consider a path $\alpha=(v_1,...,v_n)$ in $\hat{\Gamma}$ that starts with a point $v_1$ and ends with a point $v_n$.
Define its lift $\tilde{\alpha}=(\tilde{v}_1,...\tilde{v}_{k_n})$ in the Cayley graph of $\Gamma$ as follows.
Start with $\tilde{v}_1=v_1$.
If $v_2$ satisfies that $v_1^{-1}v_2\in S$, define $\tilde{v_2}=v_2$. Otherwise, $\tilde{v}_2\in v_1\mathcal{H}$, for some parabolic subgroup $\mathcal{H}$.
Choose then a geodesic between $v_1$ and $v_2$ and denote its elements by $\tilde{v_2},..,\tilde{v}_{k_2}$, so that $\tilde{v}_{k_2}=v_2$.
We do the same with $v_3$, that is if $v_3$ is obtained from $v_2$ by adding an element in $S$, we define $\tilde{v}_{k_2+1}=v_3$, otherwise, we choose a geodesic from $\tilde{v}_{k_2}$ to $v_3$, which is now denoted by $\tilde{v}_{k_3}$.
We keep doing this for every element $v_4,...,v_n$.
Note that according to \cite[Theorem~1.12]{DrutuSapir}, a geodesic whose endpoints are in the same coset of a parabolic subgroup stays in a fixed neighborhood of this parabolic subgroup.
Thus, in other words, the lift of the relative geodesic is obtained replacing the shortcuts in $\hat{\Gamma}$ with actual geodesics in the neighborhood of the corresponding parabolic subgroup.
Moreover, we have the following.

\begin{lemma}\label{lemmaliftgeodesic}\cite[Proposition~7.8, Corollary~7.10]{GerasimovPotyagailo}
There exist $(\lambda,c)$ and $(\eta_1,\eta_2)$ such that if $\alpha$ is a relative geodesic in $\hat{\Gamma}$, then its lift $\tilde{\alpha}$ is a $(\lambda,c)$-quasi geodesic.
Moreover, points in $\tilde{\alpha}$ that are obtained from lifting points in $\alpha$ are $(\eta_1,\eta_2)$-transition points on $\tilde{\alpha}$.
\end{lemma}

Note that the first part of the lemma (that the lift of a relative geodesic is a quasi-geodesic) was also stated by C.~Drutu and M.~Sapir, see \cite[Theorem~1.12]{DrutuSapir}.
It is also proved there that the lift of a relative geodesic stays within a bounded distance of a geodesic.
The converse is also true.
Precisely, we have the following.

\begin{lemma}\label{projectiontransitionpoints}\cite[Proposition~8.13]{Hruska}
Fix a generating set $S$.
For every large enough $\eta_1,\eta_2>0$, there exists $r\geq0$ such that the following holds.
Let $\alpha$ be a geodesic in $\Cay(\Gamma,S)$ and let $\gamma$ be an $(\eta_1,\eta_2)$-transition point on $\alpha$.
Let $\hat{\alpha}$ be a relative geodesic path with the same endpoints as $\alpha$.
Then, there exists a point $\hat{\gamma}$ on $\hat{\alpha}$ such $d(\gamma,\hat{\gamma})\leq r$.
\end{lemma}

\section{Preliminary results on the Green function}\label{Sectionpreliminaries}
In this section we give several general results that will be used along the paper.

\subsection{Combinatorial analysis of the Green derivatives}\label{SectionCombinatorialGreen}
Let $\Gamma$ be a finitely generated group.
Here, we do not need our group to be relatively hyperbolic and we do not make such an assumption.
Let $\mu$ be a probability measure on $\Gamma$ and let $R_{\mu}$ be the spectral radius of the corresponding random walk.
Again, we do not need to make assumptions such as the support to be finite or the associated random walk to be irreducible.
We will give formulae for the Green derivative 
$\frac{d}{dr}G(\gamma,\gamma'|r)$ and higher derivatives.
The first formula we get, given in the following lemma, was apparently first coined by S.~Gouëzel and S.~Lalley in \cite{GouezelLalley},
although it was already implicitly used before, for example by W.~Woess in \cite[\S~27.6]{Woess}.

\begin{lemma}\label{lemmafirstderivative}
For every $\gamma,\gamma'\in \Gamma$, for every $r\in [0,R_{\mu}]$, we have
$$\frac{d}{dr}(rG(\gamma_1,\gamma_2|r))=\sum_{\gamma\in \Gamma}G(\gamma_1,\gamma|r)G(\gamma,\gamma_2|r).$$
\end{lemma}

\begin{proof}
First, the formula makes sense at $R_{\mu}$, even if the sum diverges at $R_{\mu}$, since the coefficients of the power series $G(\cdot,\cdot|r)$ are non-negative.
We have by definition
$$rG(\gamma_1,\gamma_2|r)=\sum_{n\geq0}\mu^{*n}(\gamma_1^{-1}\gamma_2)r^{n+1},$$ so that
$$\frac{d}{dr}(rG(\gamma_1,\gamma_2|r))=\sum_{n\geq0}(n+1)\mu^{*n}(\gamma_1^{-1}\gamma_2)r^{n}.$$
The Cauchy formula for products of power series gives
\begin{align*}
    \sum_{\gamma\in \Gamma}G(\gamma_1,\gamma|r)G(\gamma,\gamma_2|r)&=\sum_{\gamma\in \Gamma}\sum_{n\geq 0}\left (\sum_{k=0}^n\mu^{*k}(\gamma_1^{-1}\gamma)\mu^{*(n-k)}(\gamma^{-1}\gamma_2)\right )r^n\\
    &=\sum_{n\geq0}\sum_{k=0}^n\left (\sum_{\gamma\in \Gamma}\mu^{*k}(\gamma_1^{-1}\gamma)\mu^{*(n-k)}(\gamma^{-1}\gamma_2)\right )r^n.
\end{align*}
For fixed $k$, decomposing a path of length $n$ from $\gamma_1$ to $\gamma_2$ according to its position at time $k$, we see that
$\sum_{\gamma\in \Gamma}\mu^{*k}(\gamma_1^{-1}\gamma)\mu^{*(n-k)}(\gamma^{-1}\gamma_2)=\mu^{*n}(\gamma_1^{-1}\gamma_2)$.
Thus,
$$\sum_{\gamma\in \Gamma}G(\gamma_1,\gamma|r)G(\gamma,\gamma_2|r)=\sum_{n\geq0}\sum_{k=0}^n\mu^{*n}(\gamma_1^{-1}\gamma_2)r^n=\sum_{n\geq0}(n+1)\mu^{*n}(\gamma_1^{-1}\gamma_2)r^n$$
which gives the desired formula.
\end{proof}

Using Lemma~\ref{lemmafirstderivative}, we have
$r^2\frac{d}{dr}(rG(\gamma_1,\gamma_2|r))=\sum_{\gamma\in \Gamma}(rG(\gamma_1,\gamma|r))(rG(\gamma,\gamma_2|r)),$
so that
\begin{align*}
    &\frac{d}{dr}\left (r^2\frac{d}{dr}(rG(\gamma_1,\gamma_2|r))\right )\\
    &=\sum_{\gamma\in \Gamma}\left ( \frac{d}{dr}(rG(\gamma_1,\gamma|r))rG(\gamma,\gamma_2|r)+rG(\gamma_1,\gamma|r)\frac{d}{dr}(rG(\gamma,\gamma_2|r))\right )\\
    &=r\sum_{\gamma\in \Gamma}\sum_{\gamma'\in \Gamma}G(\gamma_1,\gamma'|r)G(\gamma',\gamma|r)G(\gamma,\gamma_2|r)\\
    &\hspace{1cm}+r\sum_{\gamma\in\Gamma}\sum_{\gamma'\in \Gamma}G(\gamma_1,\gamma|r)G(\gamma,\gamma'|r)G(\gamma',\gamma_2|r)\\
    &=2r\sum_{\gamma,\gamma'\in \Gamma}G(\gamma_1,\gamma|r)G(\gamma,\gamma'|r)G(\gamma',\gamma_2|r).
\end{align*}

More generally, we define
$$I^{(k)}(r)=\sum_{\gamma^{(1)},...,\gamma^{(k)}\in \Gamma}G(\gamma,\gamma^{(1)}|r)G(\gamma^{(1)},\gamma^{(2)}|r)...G(\gamma^{(k-1)},\gamma^{(k)}|r)G(\gamma^{(k)},\gamma'|r).$$
We also inductively define
$$F_1(r)=\frac{d}{dr}(rG_r(\gamma,\gamma'))$$
and
$$F_k(r)=\frac{d}{dr}(r^2F_{k-1}(r)), k\geq 2.$$
We do not refer to $\gamma$ and $\gamma'$ in the notations, but $F_k(r)$ and $I^{(k)}(r)$ do depend on them.
According to the formula above, we have
$$F_2(r)=\frac{d}{dr}\left (r^2\frac{d}{dr}(rG(\gamma,\gamma'|r))\right )=2rI^{(2)}(r).$$
More generally, we have the following result.

\begin{lemma}\label{lemmageneralformuladerivatives}
For every $\gamma,\gamma'\in \Gamma$, for every $r\in [0,R_{\mu}]$,
$F_k(r)=k!r^{k-1}I^{(k)}(r)$.
\end{lemma}

\begin{proof}
We prove this by induction.
The formula is true for $k=2$, as stated above.
Assume it is true for $k$. Then,
$$F_{k+1}(r)=\frac{d}{dr}\left (r^2F_k(r)\right )=\frac{d}{dr}\left (k!r^{k+1}I^{(k)}(r)\right ),$$
so that
$$\frac{F_{k+1}(r)}{k!}=\frac{d}{dr}\sum_{\gamma^{(1)},...,\gamma^{(k)}\in\Gamma}(rG(\gamma,\gamma^{(1)}|r))...(rG(\gamma^{(k-1)},\gamma^{(k)}|r))(rG(\gamma^{(k)},\gamma'|r)).$$
Differentiating every $rG(\gamma^{(i)},\gamma^{(i+1)}|r)$ using Lemma~\ref{lemmafirstderivative}, we get, as for the case $k=2$,
$$\frac{F_{k+1}(r)}{k!}=(k+1)\sum_{\gamma^{(1)},...\gamma^{(k+1)}\in \Gamma}r^kG(\gamma,\gamma^{(1)}|r)...G(\gamma^{(k)},\gamma^{(k+1)}|r)G(\gamma^{(k+1)},\gamma'|r),$$
hence $F_{k+1}(r)=(k+1)!r^kI^{(k+1)}(r)$.
\end{proof}

We will use the following proposition later.
We denote by $G^{(k)}(\gamma,\gamma'|r)$ the $k$th derivative $\frac{d^k}{dr^k}G(\gamma,\gamma'|r)$ of the Green function at $r$ and simply by $G'(\gamma,\gamma'|r)$ its first derivative $\frac{d}{dr}G(\gamma,\gamma'|r)$.
\begin{proposition}\label{propderivativesGreensubexponential}
For every $\gamma,\gamma'\in \Gamma$ and every $r\in [1,R_{\mu}]$, $I^{(k)}(r)$ grows at most exponentially if and only if $\frac{G^{(k)}(\gamma,\gamma'|r)}{k!}$ does.
Precisely, there exist $c_1\geq0$ and $C_1>1$ such that $\frac{G^{(k)}(\gamma,\gamma'|r)}{k!}\leq c_1C_1^k$ for all $k$ if and only if there exist $c_2\geq0$ and $C_2>1$ such that $I^{(k)}(r)\leq c_2C_2^k$ for all $k$.
\end{proposition}

\begin{itemize}
    \item Note that the constants $c_i$ and $C_i$ are allowed to depend on $r$.
    \item The only reason for requiring $C_i>1$ and $r\geq1$ is technical. The goal is to have divergent geometric series below.
    In any case, if one of the upper bound is true for $C_i\geq0$, it is in particular true for some $C'_i>1$.
    \item Of course, $\frac{G^{(k)}(\gamma,\gamma'|R_{\mu})}{k!}\leq c_1C_1^k$ cannot happen, since $R_{\mu}$ is the radius of convergence of the Green function.
    However, the statement still remains valid at $r=R_{\mu}$ and we will use it in a proof by contradiction later.
\end{itemize}

\begin{proof}
Let us first differentiate $rG(\gamma,\gamma'|r)$.
For simplicity, we will write $G_r$ rather than $G(\gamma,\gamma'|r)$, and similarly for derivatives.
We have
$\frac{d}{dr}(rG_r)=G_r+rG'_r$.
Differentiating $r^2\frac{d}{dr}(rG_r)$, we thus get
$$F_2(r)=\frac{d}{dr}\left (r^2\frac{d}{dr}(rG_r)\right )=\frac{d}{dr}(r^2(G_r+rG'_r))=2rG_r+4r^2G'_r+r^3G^{(2)}_r.$$
Let us compute the next $F_k$ before giving a general formula.
We have
$$F_3(r)=\frac{d}{dr}(r^2(2rG_r+4r^2G'_r+r^3G^{(2)}_r))=6r^2G_r+18r^3G'_r+9r^4G^{(2)}_r+r^5G^{(3)}_r.$$
More generally, we have
\begin{equation}\label{generalequationFk}
    F_k(r)=\sum_{j=0}^kf_{j,k}r^{k+j-1}G^{(j)}_r,
\end{equation}
where the coefficients $f_{j,k}$ satisfy
\begin{equation}\label{equationfjk}
    f_{j,k+1}=f_{j-1,k}+(k+j+1)f_{j,k}, 1\leq j\leq k,
\end{equation}
since we first multiply $F_k(r)$ by $r^2$ before differentiating to get $F_{k+1}(r)$.
We also have $f_{0,k}=k!$ and $f_{k,k}=1$.

This shows that $r^{2k-1}G^{(k)}_r\leq F_k(r)=k!r^{k-1}I^{(k)}(r)$ according to Lemma~\ref{lemmageneralformuladerivatives}.
Thus, we have $\frac{G^{(k)}_r}{k!}\leq \frac{1}{r^k}I^{(k)}(r)$.
This proves the "if" part of the proposition.

For the "only if" part, we will have to be precise about the coefficients $f_{j,k}$.
We first prove by induction on $k$ that there exist $c\geq1$ and $C\geq1$ such that
\begin{equation}\label{subexponentialfjk}
    f_{j,k}\frac{j!}{k!}\leq cC^{j+k}.
\end{equation}
Indeed, assuming this is true for $k$, according to~(\ref{equationfjk}), we have for $j\leq k$
\begin{align*}
    f_{j,k+1}&\leq cC^{j-1+k}\frac{k!}{(j-1)!}+cC^{j+k}(k+j+1)\frac{k!}{j!}\\
    &\leq cC^{j-1+k}\frac{k!}{(j-1)!}+cC^{j+k}\frac{k+j+1}{k+1}\frac{(k+1)!}{j!}.
\end{align*}
Since $j\leq k$, we always have
$\frac{k!}{(j-1)!}\leq \frac{(k+1)!}{j!}$ and $(k+j+1)\leq 2(k+1)$.
Thus,
$$f_{j,k+1}\leq cC^{j-1+k}\frac{(k+1)!}{j!}+2cC^{j+k}\frac{(k+1)!}{j!}.$$
To conclude, it suffices to choose $c$ and $C$ such that $cC^{j-1+k}+2cC^{j+k}\leq cC^{j+k+1}$, i.e.\ $1+2C\leq C^2$, which is always possible.

Now, assuming that $\frac{G^{(k)}_r(x,y)}{k!}\leq c_1C_1^k$ for all $k$ and using Lemma~\ref{lemmageneralformuladerivatives} along with~(\ref{generalequationFk}) and~(\ref{subexponentialfjk}), we get
\begin{align*}
    r^{k-1}I^{(k)}(r)&\leq \frac{1}{k!}F_k(r)=\frac{1}{k!}\sum_{j=0}^kf_{j,k}r^{k+j-1}G^{(j)}_r
    \leq \frac{1}{k!}r^{k-1}\sum_{j=0}^kf_{j,k}r^jc_1C_1^jj!\\
    &\leq cc_1r^{k-1}C^k\sum_{j=0}^k(rCC_1)^j
    \leq r^{k-1}\frac{cc_1rCC_1}{rCC_1-1}(rC^2C_1)^k.
\end{align*}
This completes the proof.
\end{proof}

\begin{remark}
In \cite{Gouezel1} and \cite{GouezelLalley}, the authors use the incorrect formula
$$G(\gamma,\gamma'|r+\epsilon)=\sum_{k\geq0}\epsilon^k\sum_{\gamma^{(1)},...,\gamma^{(k)}\in \Gamma}G(\gamma,\gamma^{(1)}|r)...G(\gamma^{(k-1)},\gamma^{(k)}|r)G(\gamma^{(k)},\gamma'|r).$$
If this were true, this would imply that
$$\frac{d^k}{dr^k}G(\gamma,\gamma'|r)=\sum_{\gamma^{(1)},...,\gamma^{(k)}\in \Gamma}G(\gamma,\gamma^{(1)}|r)...G(\gamma^{(k-1)},\gamma^{(k)}|r)G(\gamma^{(k)},\gamma'|r)=I^{(k)}(r),$$
whereas Lemma~\ref{lemmageneralformuladerivatives} shows that the formula is a little bit more involved.
This is actually not a problem, since the authors only use the fact that $I^{(k)}$ grows at most exponentially to deduce that $\frac{G^{(k)}_r}{k!}$ grows at most exponentially, which is legitimate according to Proposition~\ref{propderivativesGreensubexponential}.
\end{remark}

\subsection{Spectral degenerescence and spectral positive-recurrence}\label{Sectionspectralpositiverecurrence}
We now consider a group $\Gamma$, hyperbolic relative to a collection of peripheral subgroups $\Omega$ and we fix a finite collection $\Omega_0=\{\mathcal{H}_1,...,\mathcal{H}_N\}$ of representatives of conjugacy classes of $\Omega$.
We assume that $\Gamma$ is non-elementary.
Let $\mu$ be a probability measure on $\Gamma$, $R_{\mu}$ the spectral radius of the $\mu$-random walk and $G(\gamma,\gamma'|r)$ the associated Green function, evaluated at $r$, for $r\in [0,R_{\mu}]$.
If $\gamma=\gamma'$, we simply use the notation $G(r)=G(\gamma,\gamma|r)=G(e,e|r)$.

\medskip
We denote by $p_k$ the first return transition kernel to $\mathcal{H}_k$.
Namely, if $h,h'\in \mathcal{H}_k$, then
$p_k(h,h')$ is the probability that the $\mu$-random walk, starting at $h$, eventually comes back to $\mathcal{H}_k$ and that its first return to $\mathcal{H}_k$ is at $h'$.
In other words,
\begin{align*}
    p_k(h,h')&=\mathbb{P}_h(\exists n\geq 1, X_n=h',X_1,...,X_{n-1}\notin \mathcal{H}_k)\\
    &=\sum_{n\geq 1}\sum_{\underset{\notin \mathcal{H}_k}{\gamma_1,...,\gamma_{n-1}}}\mu(h^{-1}\gamma_1)\mu(\gamma_1^{-1}\gamma_2)...\mu(\gamma_{n-2}^{-1}\gamma_{n-1})\mu(\gamma_{n-1}^{-1}h').
\end{align*}
More generally, for $r\in [0,R_{\mu}]$, we denote by $p_{k,r}$ the first return transition kernel to $\mathcal{H}_k$ for $r\mu$.
Precisely, if $h,h'\in \mathcal{H}_k$, then
$$p_{k,r}(h,h')=\sum_{n\geq 1}\sum_{\underset{\notin \mathcal{H}_k}{\gamma_1,...,\gamma_{n-1}}}r^n\mu(h^{-1}\gamma_1)\mu(\gamma_1^{-1}\gamma_2)...\mu(\gamma_{n-2}^{-1}\gamma_{n-1})\mu(\gamma_{n-1}^{-1}h').$$

We then denote by $p_{k,r}^{(n)}$ the convolution powers of this transition kernel, by $G_{k,r}(h,h'|t)$ the associated Green function, evaluated at $t$ and by $R_k(r)$ the associated spectral radius, that is, the radius of convergence of $t\mapsto G_{k,r}(h,h'|t)$.
For simplicity, write $R_k=R_k(R_{\mu})$.
As for the initial Green function, if $h=h'$, we will simply use the notation $G_{k,r}(t)=G_{k,r}(h,h|t)=G_{k,r}(e,e|t)$.
We first show the following lemma.
\begin{lemma}\label{lemmasameGreen}
Let $r\in [0,R_{\mu}]$. For any $k\in \{1,...,N\}$,
$$G_{k,r}(h,h'|1)=G(h,h'|r).$$
\end{lemma}

\begin{proof}
This follows from the fact that every trajectory from $h$ to $h'$ for $r\mu$ defines a trajectory from $h$ to $h'$ for $p_{k,r}$, excluding every point of the path that is not in $\mathcal{H}_k$
and every trajectory for $p_{k,r}$ is obtained in such a way.
Summing over all trajectories, the two Green functions coincide.
\end{proof}

Since $\Gamma$ is non-elementary, it contains a free group and hence is non-amenable.
It follows from a result of Guivarc'h (see \cite[p.~85, remark~b)]{Guivarch}) that $G(R_{\mu})<+\infty$ (see also \cite[Theorem~7.8]{Woess} for a stronger statement).
Thus, $G_{k,R_{\mu}}(1)<+\infty$.
In particular,
\begin{equation}\label{parabolicradiusgeq1}
   \forall k\in \{1,...,N\}, R_k\geq 1.
\end{equation}

\begin{definition}
We say that $\mu$ (or equivalently the random walk) is spectrally degenerate along $\mathcal{H}_k$ if $R_k=1$.
We say it is non-spectrally degenerate if for every $k$, $R_k>1$.
\end{definition}

This definition was introduced in \cite{DussauleGekhtman} to study the homeomorphism type of the Martin boundary at the spectral radius.
This homeomorphism type can change whether $\mu$ is spectrally degenerate or not along a parabolic subgroup.
Also, this definition does not depend on the choice of $\Omega_0$, in the sense that if $\mathcal{H}_k'$ is conjugate to $\mathcal{H}_k$ and if $\mathcal{H}_{k}$ is replaced with $\mathcal{H}_k'$ in $\Omega_0$, then $\mu$ is spectrally degenerate along $\mathcal{H}_0$ if and only if it is spectrally degenerate along $\mathcal{H}_k'$, see \cite[Lemma~2.2]{DussauleGekhtman}.

\medskip
Recall that for a parabolic subgroup $\mathcal{H}$,
$$I_\mathcal{H}^{(2)}(r)=\sum_{h,h'\in \mathcal{H}}G(e,h|r)G(h,h'|r)G(h',e|r).$$
For simplicity, once $\Omega_0$ is fixed, for every parabolic subgroup $\mathcal{H}_k\in \Omega_0$ we set $I^{(2)}_k=I^{(2)}_{\mathcal{H}_k}$.
In view of Lemma~\ref{lemmageneralformuladerivatives}, $I^{(2)}_k$ is finite if and only if $\frac{d^2}{dt^2}_{|t=1}\left (t\mapsto G_{k,R_\mu}(t)\right )$ is finite,
which obviously holds if $R_k>1$.
In particular, we have the following result.

\begin{proposition}\label{propspectraldegenerescenceimpliesfiniteGreenmoments}
With the same notations, if $\mu$ is non-spectrally degenerate, then it has finite Green moments.
\end{proposition}

We will also prove later that if $\mu$ is non-spectrally degenerate, then it is divergent, see Proposition~\ref{relationR_kG'(R)}.
In particular, we state here the following.

\begin{proposition}\label{propspectraldegenerescenceimpliesspectralpositiverecurrence}
With the same notations, if $\mu$ is non-spectrally degenerate, then it is spectrally positive-recurrent.
\end{proposition}

\subsection{Analogy with counting theorems}\label{Sectionanalogycountingtheorems}
We briefly explain here our choice of terminology and give some perspectives on Definitions~\ref{definitionGreenmoments},~\ref{definitiondivergent} and~\ref{definitionspectralpositiverecurrent} and on our main theorem.

Consider a simply connected complete Riemannian manifold $X$ with pinched negative curvature
Let $\Gamma$ be a finitely generated group acting on $X$ via a discrete and free action.
Denote by $M=X/\Gamma$ the quotient manifold.
We assume that $M$ is geometrically finite, or equivalently that $\Gamma$ is geometrically finite, see \cite{Bowditchgeometricallyfinite1} and \cite{Bowditchgeometricallyfinite2} for more details.
In particular $\Gamma$ is relatively hyperbolic with respect to virtually nilpotent parabolic subgroups.

The Poincar\'e series is defined as
\begin{equation}\label{defPoincare}
P_{\Gamma}(s)=\sum_{\gamma\in \Gamma}\mathrm{e}^{-sd(x_0,\gamma\cdot x_0)}.
\end{equation}
The critical exponent $\delta_{\Gamma}$ of the group is the exponential radius of convergence of this series.
Precisely, $P_{\Gamma}(s)$ is finite if $s>\delta_{\Gamma}$ and $P_{\Gamma}(s)$ is infinite if $s<\delta_{\Gamma}$.

Following S.~Patterson and D.~Sullivan \cite{Patterson}, \cite{Sullivan}, we say that $\Gamma$ is divergent if $P_{\Gamma}(\delta_{\Gamma})=+\infty$ and that $\Gamma$ is convergent otherwise.

\medskip
We also introduce the series of moments along a parabolic subgroup $\mathcal{H}$ as
\begin{equation}\label{defseriesofmoment}
    M_{\mathcal{H}}(s)=\sum_{\gamma\in \Gamma}d(x_0,\gamma \cdot x_0)\mathrm{e}^{-sd(x_0,\gamma\cdot x_0)}.
\end{equation}
Note that the series of moments is the derivative of the Poincar\'e series of $\mathcal{H}$.
We say that $\Gamma$ is positive-recurrent if it is divergent and if for every parabolic subgroup $\mathcal{H}$, $M_{\mathcal{H}}$ is finite.
Positive-recurrence is a key property for establishing counting theorems and dynamical properties of $\Gamma$ in such contexts.
Note for example that according to \cite[Theorem~B]{DalboOtalPeigne}, $\Gamma$ is positive-recurrent if and only if the associated Bowen-Margulis measure $m_\Gamma$ is finite.

\medskip
We can also introduce the critical exponent $\delta_{\mathcal{H}}$ of the parabolic subgroup $\mathcal{H}$ as the exponential radius of convergence of the Poincar\'e series restricted to $\mathcal{H}$.
We say that $\Gamma$ has a spectral gap if for every parabolic subgroup $\mathcal{H}$, $\delta_{\mathcal{H}}<\delta_\Gamma$.

Clearly, if $\Gamma$ has a spectral gap, then $M_\mathcal{H}$ is finite for every $\mathcal{H}$.
It is also true that if $\Gamma$ has a spectral gap, then it is divergent, see \cite[Theorem~A]{DalboOtalPeigne}.
We also refer to \cite{Roblin} and references therein for many more details.

\medskip
Let us now make a formal analogy with random walks on relatively hyperbolic groups.
Let $\mu$ be an admissible probability measure on a finitely generated group $\Gamma$ defining a transient random walk (that is, for every $\gamma$, $G(e,\gamma)$ is finite).
Following Brofferio and Blach\`ere \cite{BlachereBrofferio}, we introduce the Green distance as
$$d_G(\gamma,\gamma')=-\log F(\gamma,\gamma')=-\log \frac{G(\gamma,\gamma')}{G(e,e)}.$$
Then, $F(\gamma,\gamma')$ is the probability of ever reaching $\gamma'$, starting the random walk at $\gamma$, see \cite[Lemma~1.13~(b)]{Woess}.
When the measure $\mu$ is symmetric, this is indeed a distance.
The triangle inequality can be reformulated as
$$F(\gamma_1,\gamma_2)F(\gamma_2,\gamma_3)\leq F(\gamma_1,\gamma_3).$$
In particular, for any $\gamma_1,\gamma_2,\gamma_3$, we have
\begin{equation}\label{triangleGreen}
G(\gamma_1,\gamma_2)G(\gamma_2,\gamma_3)\leq CG(\gamma_1,\gamma_3)
\end{equation}
for some uniform constant $C$.
Note that this inequality is always true, whether $\mu$ is symmetric or not, since it only states that the probability of reaching $\gamma_3$ starting at $\gamma_1$ is always bigger than the probability of first reaching $\gamma_2$ from $\gamma_1$ and then reaching $\gamma_3$ from $\gamma_2$.

More generally, we set
\begin{equation}\label{defF}
    F(\gamma,\gamma'|r)=\frac{G(\gamma,\gamma'|r)}{G(e,e|r)}.
\end{equation}
We also define the $r$-Green metric as $d_{G}(\gamma,\gamma|r)=-\log F(\gamma,\gamma'|r)$
and the symmetrized $r$-Green metric as
$$\tilde{d}_G(\gamma,\gamma'|r)=-\log F(\gamma,\gamma'|r)-\log F(\gamma',\gamma|r).$$
Notice then that
$$I^{(1)}(r)=\sum_{\gamma\in \Gamma}G(e,\gamma|r)G(\gamma,e|r)=\sum_{\gamma\in \Gamma}\mathrm{e}^{-\tilde{d}_G(e,\gamma|r)}.$$
Thus, $I^{(1)}(r)$ is analogous to the Poincar\'e series associated with the symmetrized Green metric.
The main difference with~(\ref{defPoincare}) is that the parameter $r$ is part of the definition of the metric in our situation.
Lemma~\ref{lemmafirstderivative} shows that $I^{(1)}(r)$ is finite if and only if $\frac{d}{dr}G(e,e|r)$ is finite.
In particular, we see that for $r<R_\mu$, $I^{(1)}(r)$ is finite and for $r>R_\mu$, $I^{(1)}(r)$ is infinite.
Thus, under this analogy, the spectral radius of the random walk should be compared with the critical exponent of the group.
Moreover, we see that our terminology \textit{divergent} and \textit{convergent} is coherent with that of Patterson and Sullivan.

\medskip
As explained, given a parabolic subgroup $\mathcal{H}$, the series of moments defined by~(\ref{defseriesofmoment}) is the derivative of the Poincar\'e series of $\mathcal{H}$.
In our situation, we thus have to check finiteness of $\frac{d}{dr}I^{(1)}(r)$.
According to Lemma~\ref{lemmageneralformuladerivatives} this derivative is finite if and only if $\frac{d^2}{dr^2}G(e,e|r)$ is finite, or equivalently $I^{(2)}_{\mathcal{H}}(r)$ is finite.
This explains our definition of \textit{having finite Green moments}.
Finally, \textit{being spectrally positive-recurrent} is formally analogous to \textit{being positive-recurrent} for a geometrically finite group, as defined above.

The fact that this property of being positive-recurrent seems so important for counting theorems and is in particular equivalent to having a finite Bowen-Margulis measure also gives motivation for our main theorem.

\medskip
Finally, we see that under this analogy, the fact that the measure is non-spectrally degenerate should be compared with the fact that group has a spectral gap.
As noted above, in both situations, this implies being (spectrally) positive-recurrent.

\section{Coding relatively hyperbolic groups}\label{Sectioncoding}
In all this subsection, finite generating sets $S$ of a group $\Gamma$ are assumed to be symmetric for simplicity.
Hyperbolic groups are known to be strongly automatic, meaning that for every such generating set $S$, there exists a finite directed graph $\mathcal{G}=(V,E,v_*)$
with a distinguished vertex $v_*$ called the starting vertex and with a labelling map $\phi:E\rightarrow S$ such that the following holds.
If $\gamma=e_1...e_n$ is a path of adjacent edges in $\mathcal{G}$, define $\phi(e_1...e_m)=\phi(e_1)...\phi(e_m)\in \Gamma$.
The properties satisfied by the labelled graph are the following.
No edge ends at $v_*$, every vertex $v\in V$ can be reached from $v_*$ in $\mathcal{G}$, for every path $\gamma=e_1...e_n$, the path $e,\phi(e_1),\phi(e_1e_2),...,\phi(\gamma)$ in $\Gamma$ is a geodesic from $e$ to $\phi(\gamma)$ for the word distance corresponding to the generating set $S$ and
the extended map $\phi$ is a bijection between paths in $\mathcal{G}$ starting at $v_*$ and elements of $\Gamma$.
We refer to \cite[Chapter~9,~Theorem~13]{GhysHarpe} for a proof of this fact.
This result was first proved by J.~Cannon in \cite{Cannon} for some particular negatively curved groups.

Automatic structures for relatively hyperbolic have been explored by D.~Rebbechi in his thesis \cite{Rebbechithesis} and then by Y.~Antol\'in and L.~Ciobanu in \cite{AntolinCiobanu}.
Namely, D.~Rebbechi proved that if $\Gamma$ is hyperbolic relative to automatic groups, then, it is itself automatic.
Y.~Antol\'in and L.~Ciobanu gave a strengthened version of this result (see \cite[Theorem~1.1]{AntolinCiobanu}).
Let us also mention the work of W.~Neumann and M.~Shapiro on geometrically finite groups (see \cite{NeumannShapiro}).
In the following, we will have to deal with what we call relative automatic structures.
Unfortunately, we can neither use directly the results of D.~Rebbechi nor those of Y.~Antol\'in and L.~Ciobanu.

Let $\Gamma$ be a finitely generated group and let $\Omega$ be a collection of subgroups invariant by conjugacy and such that there is a finite set $\Omega_0$ of conjugacy classes representatives of subgroups in $\Omega$.

\begin{definition}\label{definitionautomaticstructure}
A relative automatic structure for $\Gamma$ with respect to the collection of subgroups $\Omega_0$ and with respect to some finite generating set $S$ is a directed graph $\mathcal{G}=(V,E,v_*)$ with distinguished vertex $v_*$ called the starting vertex, where the set of vertices $V$ is finite and with a labelling map $\phi:E\rightarrow S\cup \bigcup_{\mathcal{H}\in \Omega_0}\mathcal{H}$ such that the following holds.
If $\omega=e_1,...,e_n$ is a path of adjacent edges in $\mathcal{G}$, define $\phi(e_1,...,e_n)=\phi(e_1)...\phi(e_n)\in \Gamma$.
Then,
\begin{itemize}
    \item no edge ends at $v_*$, except the trivial edge starting and ending at $v_*$,
    \item every vertex $v\in V$ can be reached from $v_*$ in $\mathcal{G}$,
    \item for every path $\omega=e_1,...,e_n$, the path $e,\phi(e_1),\phi(e_1e_2),...,\phi(\gamma)$ in $\Gamma$ is a relative geodesic from $e$ to $\phi(\gamma)$, that is the image of $e,\phi(e_1),\phi(e_1e_2),...,\phi(\gamma)$ in $\hat{\Gamma}$ is a geodesic for the metric $\hat{d}$,
    \item the extended map $\phi$ is a bijection between paths in $\mathcal{G}$ starting at $v_*$ and elements of $\Gamma$.
\end{itemize}
\end{definition}

Note that the union $S\cup \bigcup_{\mathcal{H}\in \Omega_0}\mathcal{H}$ is not required to be a disjoint union.
Actually, the intersection of two distinct subgroups $\mathcal{H},\mathcal{H}'\in \Omega_0$ can be non-empty.
Also note that we require the vertex set $V$ to be finite.
However, the set of edges is infinite, except if the parabolic subgroups $\mathcal{H}$ are finite (in which case the group $\Gamma$ is hyperbolic).

If there exists a relative automatic structure for $\Gamma$ with respect to $\Omega_0$ and $S$, we say that $\Gamma$ is automatic relative to $\Omega_0$ and $S$.
The goal of this subsection is to prove the following theorem.

\begin{theorem}\label{thmcodingrelhypgroups}
Let $\Gamma$ be a relatively hyperbolic group and let $\Omega_0$ be a finite set of representatives of conjugacy classes of the maximal parabolic subgroups.
For every symmetric finite generating set $S$ of $\Gamma$,
$\Gamma$ is automatic relative to $\Omega_0$ and $S$.
\end{theorem}

During the proof, we will be careful about the notations for the distance. The letter $d$ will stand for the distance in the Cayley graph for the generating set $S$ whereas the letter $\hat{d}$ will refer to the graph distance in $\hat{\Gamma}$. We will use both.

We will use several times the following classical results.
Recall that a $k$-local relative geodesic is a path $(x_1,...,x_n)$ such that every subpath of length $k$ is a relative geodesic.

\begin{lemma}\cite[Theorem~5.2]{AntolinCiobanu}\label{Theorem52AntolinCiobanu}
There exist $\lambda$ and $c$ and a finite set of non-geodesic sequences $\mathcal{NG}$ of the form $\overline{\sigma}=(\sigma_1,...,\sigma_n)$, where $\sigma_i\in S\cup \bigcup \mathcal{H}_k$, such that every 2-local relative geodesic $(x_1,...,x_n)$ that does not contain any sequence of $\mathcal{NG}$ as a subpath, is a relative $(\lambda,c)$-quasi geodesic path.
\end{lemma}

\begin{lemma}\cite[Corollary~8.15]{Hruska}\label{Proposition315Osin}
For every $\lambda\geq1,c\geq0$ and $K\geq 0$, there exists $C\geq 0$ such that the following holds.
Let $(x_1,...,x_n)$ and $(x_1',...,x_m')$ be two relative $(\lambda,c)$-quasi geodesic paths such that $d(x_1,x'_1)\leq K$ and $d(x_n,x'_m)\leq K$.
Assume that these quasi-geodesic are without backtracking.
Then for every $1\leq j\leq n$ such that $x_j$ is the entrance or exit point of a parabolic subgroup, there exists $i_j$ such that $d(x_j,x_{i_j}')\leq C$.
In particular, if $(x_1,...,x_m)$ is a relative geodesic, then for every $j$, there exists such an $i_j$.
\end{lemma}

Actually, \cite[Corollary~8.15]{Hruska} is only about relative geodesics, but one deduces the result for relative quasi-geodesic paths using the BCP property.
Also note that this lemma is exactly the content of \cite[Proposition~3.15]{Osin}, although the author uses a particular generating set there.
We also have a version of this lemma for relative geodesic rays.

\begin{lemma}\label{Proposition315Osininfinite}
For every $(\lambda,c)$ and $K\geq 0$, there exists $C\geq 0$ such that the following holds.
Let $(x_1,...,x_n,...)$ and $(x_1',...,x_m',...)$ be two infinite relative $(\lambda,c)$-quasi-geodesic paths such that
$d(x_1,x'_1)\leq K$ and $x_n$ and $x_m'$ converge to the same conical limit point $\xi$.
Assume that these relative geodesics are withouth backtracking.
Then for every $1\leq j\leq n$ such that $x_j$ is the entrance or exit point of a parabolic subgroup, there exists $i_j$ such that $d(x_j,x_{i_j}')\leq C$.
In particular, if $(x_1,...,x_n,...)$ is a relative geodesic, then for every $j$, there exists such an $i_j$.
\end{lemma}

\begin{proof}
Consider two such $(\lambda,c)$-quasi-geodesic paths.
Convergence to a conical limit point means convergence to the Gromov boundary of $\hat{\Gamma}$.
Since $\hat{\Gamma}$ is hyperbolic, there exists $C_1$ such that for every $j$, there exists $k_j$ such that $\hat{d}(x_j,x'_{k_j})\leq C_1$.
Fix $j$ and consider $x_j$ and $x_{k_j}'$ as above.
Then, the concatenation of the relative quasi-geodesic path $(x_1,...,x_j)$ and a relative geodesic from $x_j$ to $x_{k_j}'$ is a relative $(\lambda',c')$-quasi-geodesic path with fixed parameters $\lambda',c'$ (that only depend on $C_1,\lambda,c$).
Letting $\lambda''=\max (\lambda,\lambda')$ and $c''=\max (c,c')$, we thus have two relative $(\lambda'',c'')$-quasi-geodesic paths starting at $x_1$ and $x_1'$ and ending at $x_{k_j}'$.
Applying Lemma~\ref{Proposition315Osin} to those two paths, we see that there exists $i_j'$ such that $d(x_j,x_{i_j}')\leq C$ for some $C$ that only depends on $C_1,\lambda'',c''$, thus only on $\lambda,c$.
\end{proof}

We will adapt the proof of \cite[Chapter~9,~Theorem~13]{GhysHarpe} to the relative case to prove Theorem~\ref{thmcodingrelhypgroups}.
We denote by $\Sigma$ the set $S\cup \bigcup_{\mathcal{H}\in \Omega_0}\mathcal{H}$ and by $\Sigma^*$ the complete language over the alphabet $\Sigma$, that is, the set of all finite sequences of elements of $\Sigma$.
We denote by $\overline{\sigma}=(\sigma_1,\sigma_2,...,\sigma_n)$ a sequence in $\Sigma^*$ and we define then the elements of $\Gamma$ $\gamma_1=\sigma_1$, $\gamma_2=\sigma_1\sigma_2$, ..., $\gamma_n=\sigma_1...\sigma_n$.
We denote by $\overline{\gamma}=(e,\gamma_1,...,\gamma_n)$ the sequence of elements of $\Gamma$ thus produced (notice that we added the neutral element $e$ to the sequence).
Say that $\overline{\sigma}$ is \emph{reduced} if $\overline{\gamma}$ is a relative geodesic from $e$ to $\gamma_n$.
From the definition of the $\hat{\Gamma}$, we get the following.

\begin{lemma}
Every $\gamma$ can be represented by a reduced sequence $\overline{\sigma}=(\sigma_1,\sigma_2,...,\sigma_n)$ in $\Sigma^*$, meaning that $\gamma=\sigma_1\sigma_2...\sigma_n$.
\end{lemma}

%\begin{proof}
%Recall that $(\hat{\Gamma},\hat{d})$ is a geodesic metric space.
%Let $\gamma\in \Gamma$.
%Then, there exists a geodesic from the image of $e$ to the image of $\gamma$ in $\hat{\Gamma}$.
%Denote by $v_0,v_1,...,v_n$ this geodesic, where $v_0$ is the image of $e$ and $v_n$ is the image of $\gamma$.
%Recall that the edges of $\hat{\Gamma}$ are the edges of the Cayley graph constructed with $S$ together with the additional edges between elements of the parabolic subgroups to the corresponding additional vertex.
%Thus, for $1\leq k \leq n$, if $v_k$ and $v_{k-1}$ are images of elements $\gamma_k,\gamma_{k-1}$ of $\Gamma$, we have $\gamma_k=\gamma_{k-1}s$ for some $s\in S$.
%On the contrary, if $v_k$ is one additional vertex for some parabolic subgroup $H$, then both $v_{k-1}$ and $v_{k+1}$ are images of elements $\gamma_{k-1}$ and $\gamma_{k+1}$ of $H$, so that $\gamma_{k+1}=\gamma_kh$ for some $h\in H$.
%Finally, removing every vertex $v_k$ of this form, we get a sequence $\gamma_1,...\gamma_l$ with $l\leq n$ and $\gamma_0=e$, $\gamma_l=\gamma$, whose image in $\hat{\Gamma}$ is a geodesic for $\hat{d}$ and such that $\gamma_k=\gamma_{k-1}\sigma$ for some $\sigma \in \Sigma$.
%\end{proof}

However, such a sequence is not unique in general.
We will first construct an automaton that will recognize every reduced sequence.
This will not prove the theorem, since the map $\phi$ will not be bijection, but we will then modify the automaton to have this property.
Denote by $\mathcal{R}\subset \Sigma^*$ the subset of reduced sequences.

\begin{definition}
We will say that two reduced sequences $\overline{\sigma}$ and $\overline{\sigma}'$ have the same relative cone-type if for every reduced sequence $\overline{\sigma}''$, the concatenation of $\overline{\sigma}$ and $\overline{\sigma}''$ is again reduced if and only if the concatenation of $\overline{\sigma}'$ and $\overline{\sigma}''$ also is reduced.
In other words, the reduced sequences extending $\overline{\sigma}$ and $\overline{\sigma}'$ to reduced sequences are the same.
\end{definition}

Having the same relative cone-type is an equivalence relation. We denote by $V_0$ the quotient set and define a graph $\mathcal{G}_0=(V_0,E_0,v_0)$ with vertex set $V_0$, distinguished vertex $v_0$ the cone-type of the empty sequence and edges set $E_0$ constructed as follows.
If $v\in V_0$, we choose a sequence $\overline{\sigma}$ with relative cone-type $v$.
For every reduced sequence $\overline{\sigma}'$ that is obtained from $\overline{\sigma}$ by adding some $\sigma\in \Sigma$, we introduce an edge $\overline{e}$ in $\mathcal{G}_0$
from $v$ to the cone-type of $\overline{\sigma}'$.
We label this edge by $\phi_0(\overline{e}):=\sigma$.
This construction does not depend on the choice of $\overline{\sigma}$.

First, we have the following result.
\begin{lemma}\label{lemmasametypeempty}
The empty sequence is the only one having its relative cone-type.
\end{lemma}

\begin{proof}
For any $\sigma\in \Sigma\setminus \{e\}$, the sequence only consisting of $\sigma$ extends the empty sequence.
On the contrary, if $\overline{\sigma}=(\sigma_1,...,\sigma_n)$, the sequence $(\sigma_1,...,\sigma_n,\sigma_n^{-1})$ is not reduced (and $\sigma_n^{-1}$ does lie in $\Sigma$ since the generating set $S$ is assumed to be symmetric).
Thus, the sequence only consisting of $\sigma_n^{-1}$ does not extend $\overline{\sigma}$.
\end{proof}

We also have the following proposition, which is basically a consequence of the BCP property, although the actual proof is a bit involved.
See also \cite[Lemma~B.1]{Yang2}
for a similar statement in a different setting.

\begin{proposition}\label{finiteconetypes}
There is only a finite number of different cone-types of sequences. In other words, the set $V_0$ of vertices of $\mathcal{G}_0$ is finite.
\end{proposition}

The proposition will be proved in several steps.
The first lemma states that concatenating a relative geodesic that ends with a long word in a parabolic subgroup with another relative geodesic which starts with the same word again yields a relative geodesic.

\begin{lemma}\label{relativegeodesiclongdistances1}
There exists $C_0\geq 0$ such that the following holds.
Let $y\in \Sigma$ with $d(e,y)\geq C_0$.
Let $\alpha=(x_0,...,x_n,x_ny)$ and $\beta=(y,z_1,...,z_m)$ be two relative geodesic paths whose first and last jumps respectively are $y$.
Then, the concatenation of $\alpha$ and $\beta$ $(x_0,x_1,x_2,...,x_n,x_ny,x_nz_1,...,x_nz_m)$ is a relative geodesic.
\end{lemma}

\begin{proof}
If $d(e,y)$ is larger than the maximum of $d(e,s)$ for $s\in S$, we know that $y$ lies in some parabolic subgroup $\mathcal{H}$.
Since $(x_0,...,x_n,x_ny)$ and $(y,z_1,...,z_m)$ are relative geodesics, we can apply Lemma~\ref{Theorem52AntolinCiobanu} to $(x_0,x_1,x_2,...,x_n,x_ny,x_nz_1,...,x_nz_m)$.
Thus, it is a relative $(\lambda,c)$-quasi geodesic path with fixed parameters $\lambda,c$.

Let $C_0$ be larger than the constant $C_{\lambda,c}$ in the BCP property for $\lambda$ and $c$.
Then, any relative geodesic from $x_0$ to $x_nz_m$ has to pass through the coset $x_n\mathcal{H}$, precisely through some points $\tilde{x}_n$ and $\tilde{y}$
with $d(x_n,\tilde{x}_n)\leq C_{\lambda,c}$ and $d(x_ny,\tilde{y})\leq C_{\lambda,c}$.
If $C_0$ is large enough, then one necessarily has $\tilde{x}_n\neq \tilde{y}$.
Consider such a relative geodesic.

Assume by contradiction that $(x_1,x_2,...,x_n,x_ny,x_nz_1,...,x_nz_m)$ is not a relative geodesic.
Then, the length of our relative geodesic from $x_0$ to $x_nz_m$ is smaller than $n+m+1$.
Thus, either the sub-geodesic from $x_0$ to $\tilde{x}_n$ has length smaller than $n$, or the sub-geodesic from $\tilde{y}$ to $x_nz_m$ has length smaller than $m$.
In the first case, one get a path from $x_0$ to $x_ny$ by adding one edge from $\tilde{x}_n$ to $x_ny$ (recall that they lie in the same coset) that has length smaller or equal to $n$.
This is a contradiction, since $(x_0,...,x_n,x_ny)$ is a relative geodesic.
The second case similarly leads to a contradiction.
\end{proof}

In the same spirit, we also prove the following.

\begin{lemma}\label{relativegeodesiclongdistances2}
There exists $C_1\geq 0$ such that the following holds.
Let $y$ and $y'$ be two elements of the same parabolic subgroup $\mathcal{H}\in \Omega_0$, with $d(e,y),d(e,y')\geq C_1$.
Let $(x_0,x_1,...,x_n)$ be a relative geodesic.
Then, $(x_0,...,x_n,x_ny)$ is a relative geodesic if and only if $(x_0,...,x_n,x_ny')$ also is one.
\end{lemma}

\begin{proof}
Assume that $(x_0,...,x_n,x_ny)$ is a relative geodesic.
Then, since $\hat{d}(y,y')=1$, $(x_0,x_1,...,x_n,x_ny')$ is a $(\lambda,c)$-quasi geodesic path for some fixed $\lambda$ and $c$.
Let $C_1$ be larger than the constant $C_{\lambda,c}$ in the BCP property.
Assume by contradiction that $(x_0,...,x_n,x_ny')$ is not a relative geodesic
and take a relative geodesic from $x_0$ to $x_ny'$.
According to the BCP property, it passes through some point $\tilde{x}_n$ in the coset $x_n\mathcal{H}$, with $d(x_n,\tilde{x}_n)\leq C_1$.
In particular, $\tilde{x}_n\neq x_ny$.
Moreover, this geodesic has length smaller than $n+1$, so that its sub-geodesic from $e$ to $\tilde{x}_n$ has length smaller than $n$.
Thus, adding a vertex between $\tilde{x}_n$ and $x_ny$ (recall that they lie in the same coset) gives a path from $x_0$ to $x_ny$ of length smaller than $n+1$, which is a contradiction.
\end{proof}

For a fixed constant $C\geq 0$, denote by $B_C$ the ball of radius $C$ and center $e$ in the Cayley graph endowed with the word distance for the generating set $S$.
Define then, for $\gamma\in \Gamma$,
$$\hat{\rho}_C(\gamma):g\in B_C\mapsto \hat{d}(e,\gamma g)-\hat{d}(e,\gamma)\in \R.$$
We insist on the fact that we use the ball in the Cayley graph and not the relative ball in $\hat{\Gamma}$ in the definition of $\hat{\rho}$.
In particular, $B_C$ is finite, this will be a crucial point in the proof of Proposition~\ref{finiteconetypes}.

\begin{lemma}\label{relativegeodesiclongdistances3}
There exist $n_0\geq 0$ and $C_2\geq 0$ such that the following holds.
Let $(x_0,...,x_m)$ and $(x_0',...,x_l')$ be relative geodesics and $(e,z_1,...,z_j)$ be another relative geodesic, with $j\geq n_0$.
Assume that both $(x_0,...,x_m,x_mz_1,...,x_mz_j)$ and $(x_0',...,x_l',x_l'z_1,...,x_l'z_{n_0})$ are relative geodesics.
Assume that setting $\gamma=x_0...x_m$ and $\gamma'=x_0'...x_l'$, we have
$\hat{\rho}_{C_2}(\gamma)=\hat{\rho}_{C_2}(\gamma')$.
Then $(x_0',...,x_l',x_l'z_1,...,x_l'z_j)$ is a relative geodesic.
\end{lemma}

\begin{proof}
Let $\gamma_0=x_l'z_j$.
Consider a geodesic $\alpha=(e,\alpha_1,...\alpha_n)$ from $e$ to $\gamma_0$, of length $n\leq l+j$.
We have to prove that $n\geq l+j$.
We already know that $(x_0',...,x_l',x_l'z_1,...,x_l'z_{n_0})$ is a relative geodesic, so $(x_0',...,x_l',x_l'z_1,...,x_l'z_j)$ is a $k$-local relative geodesic, where $k$ only depends on $n_0$ and tends to infinity as $n_0$ tends to infinity.
Thus, \cite[Section~3, Theorem~1.4]{CoornaertDelzantPapadopoulos} shows that
if $n_0$ is large enough,
$(x_0',...,x_l',x_l'z_1,...,x_l'z_j)$ is a relative $(\lambda,c)$-quasi geodesic path for some $\lambda$ and $c$.
Thus, Lemma~\ref{Proposition315Osin} shows that there exists $p$ such that
$d(\gamma',\alpha_p)\leq C$, for some $C\geq 0$.

Consider multiplication on the left by $\gamma\gamma'^{-1}$. It sends $\gamma'$ to $\gamma$,
$\gamma_0$ to $x_mz_j$ and $\alpha_p$ to some $\beta_p$.
Since $\Gamma$ acts by isometries, both on its Cayley graph and on $\hat{\Gamma}$,
the relative geodesic $\alpha$ is sent to a relative geodesic and $d(\gamma,\beta_p)\leq C$.

Let $C_2\geq C$, so that if $\hat{\rho}_{C_2}(\gamma)=\hat{\rho}_{C_2}(\gamma')$,
then
$$\hat{d}(e,\alpha_p)-\hat{d}(e,\gamma')=\hat{d}(e,\beta_p)-\hat{d}(e,\gamma),$$
which can be written as
\begin{equation}\label{equationGhysHarpe1}
    \hat{d}(e,\beta_p)=m+p-l.
\end{equation}
Moreover,
$\hat{d}(\beta_p,\gamma)=\hat{d}(\alpha_p,\gamma')$ and since $\alpha$ is a relative geodesic of length $n$,
we get
\begin{equation}\label{equationGhysHarpe2}
    \hat{d}(\beta_p,\gamma)=n-p.
\end{equation}

Finally, $(x_0,...,x_m,x_mz_1,...,x_mz_j)$ is a relative geodesic, so that the triangle inequality gives
$m+j=\hat{d}(e,\gamma)\leq \hat{d}(e,\beta_p)+\hat{d}(\beta_p,\gamma)$
so that, using~(\ref{equationGhysHarpe1}) and~(\ref{equationGhysHarpe2}),
$$m+j\leq m+p-l+n-p=m+l-n,$$
which shows that $l+j\leq n$.
\end{proof}

\begin{lemma}\label{samefunctionsameconetype}
There exists $C\geq 0$ such that the following holds.
Let $\overline{\sigma}=(\sigma_1,...,\sigma_n)$ and $\overline{\sigma}'=(\sigma_1',...\sigma_m')$ be two reduced sequences. Let $\gamma=\sigma_1...\sigma_n$ and $\gamma'=\sigma_1'...\sigma_m'$ in $\Gamma$.
Assume that $\hat{\rho}_C(\gamma)=\hat{\rho}_C(\gamma')$.
Then, $\overline{\sigma}$ and $\overline{\sigma}'$ have the same relative cone-type.
\end{lemma}

\begin{proof}
Let $C_3$ be larger than $C_0$, $C_1$ and $C_2$ from Lemmas~\ref{relativegeodesiclongdistances1},~\ref{relativegeodesiclongdistances2} and~\ref{relativegeodesiclongdistances3} and let $C=n_0C_3$.
Consider a reduced sequence $(\sigma_1'',...,\sigma_l'')$ and define as usual the points $\gamma_1''=\sigma_1''$,...,$\gamma_l''=\sigma_1''...\sigma_l''$.
Assume then that
$(e,\gamma_1,...,\gamma_n,\gamma_n\gamma_1'',...,\gamma_n\gamma_l'')$ is a relative geodesic, where $\gamma_1=\sigma_1$,...,$\gamma_n=\sigma_1...\sigma_n$.
We want to prove that
$(e,\gamma_1',...,\gamma_m',\gamma_m'\gamma_1'',...,\gamma_m'\gamma_l'')$ also is a relative geodesic.

First, if for all $j\leq n_0$, $d(\gamma_{j-1}'',\gamma_j'')\leq C_3$, then
$(e,\gamma_1',...,\gamma_m',\gamma_m'\gamma_1'',...,\gamma_m'\gamma_{n_0}'')$ is a relative geodesic, since $\hat{\rho}_C(\gamma)=\hat{\rho}_C(\gamma')$.
Thus, $(e,\gamma_1',...,\gamma_m',\gamma_m'\gamma_1'',...,\gamma_m'\gamma_l'')$ is indeed a relative geodesic, according to Lemma~\ref{relativegeodesiclongdistances3}. 

On the contrary, assume that for some $j\leq n_0$, $d(\gamma_{j-1}'',\gamma_j'')>C_3$ and let $j_0$ be the smallest of such $j$.
In particular, $(\gamma_{j_0-1}'')^{-1}\gamma_{j_0}''$ is in some parabolic subgroup, say $\mathcal{H}$.
Let $\sigma$ be in the same parabolic subgroup $\mathcal{H}$, with $C_1\leq d(e,\sigma)\leq C_3$.
We know that $(e,\gamma_1,...,\gamma_n,\gamma_n\gamma_1'',...,\gamma_n\gamma_{j_0}'')$ is a relative geodesic, so according to Lemma~\ref{relativegeodesiclongdistances2}, $(e,\gamma_1,...,\gamma_n,\gamma_n\gamma_1'',...,\gamma_n\gamma_{j_0-1}'',\gamma_n\gamma_{j_0-1}''\sigma)$ also is one.
Since $\hat{\rho}_{C}(\gamma)=\hat{\rho}_{C}(\gamma')$ and $d(e,h)\leq C_3$, $(e,\gamma_1',...,\gamma_m',\gamma_m'\gamma_1'',...,\gamma_m'\gamma_{j_0-1}'',\gamma_m'\gamma_{j_0-1}''\sigma)$ also is one.
Using again Lemma~\ref{relativegeodesiclongdistances2}, we see that $(e,\gamma_1',...,\gamma_m',\gamma_m'\gamma_1'',...,\gamma_m'\gamma_{j_0}'')$ also is a relative geodesic.

Finally, since $C_3\geq C_0$, Lemma~\ref{relativegeodesiclongdistances1} shows that $(e,\gamma_1',...,\gamma_m',\gamma_m'\gamma_1'',...,\gamma_m'\gamma_l'')$ also is a relative geodesic.
\end{proof}

Proposition~\ref{finiteconetypes} now follows from Lemma~\ref{samefunctionsameconetype}.
Indeed, since $B_C$ is finite, there is a finite number of different functions $\hat{\rho}_C(\gamma)$. \qed

\medskip
Lemma~\ref{lemmasametypeempty} shows that the graph $\mathcal{G}_0$ satisfies the first condition in Definition~\ref{definitionautomaticstructure} and it also satisfies the second and third conditions by definition.
Also, Proposition~\ref{finiteconetypes} shows that its set of vertices is finite.
As announced, we will modify $\mathcal{G}_0$ so that it also satisfies the fourth one.
We arbitrarily choose an order on the countable set $\Sigma$ and endow $\Sigma^*$ with the associated lexicographical order, that we denote by $\leq$.
We will say that a reduced sequence $\overline{\sigma}=(\sigma_1,...,\sigma_n)$ is nicely reduced if for all other reduced sequence $\overline{\sigma}'=(\sigma_1',...,\sigma_m')$ satisfying $\sigma_1...\sigma_n=\sigma_1'...\sigma_m'$ in $\Gamma$, $\overline{\sigma}\leq \overline{\sigma}'$.
We denote by $\mathcal{N}\subset \mathcal{R}$ the set of nicely reduced sequences.
The map $\overline{\sigma}=(\sigma_1,...,\sigma_n)\in \mathcal{N}\mapsto \sigma_1...\sigma_n\in \Gamma$ is now a bijection.
Our goal is thus to modify $\mathcal{G}_0$ so that the accepted language is $\mathcal{N}$ rather than $\mathcal{R}$.

\medskip
Let $\overline{\sigma}=(\sigma_1,...,\sigma_n)\in \mathcal{R}$ be reduced and let $C\geq 0$.
Let $\overline{\sigma}'=(\sigma_1',...,\sigma_n')$ be a sequence in $\Sigma^*$ with the same number of elements.
Let $\overline{\gamma}$ and $\overline{\gamma}'$ be the corresponding sequences of elements of $\Gamma$.
Say that $\overline{\sigma}'$ relatively $C$-follows $\overline{\sigma}$ if for every $1\leq k\leq n$,
$\hat{d}(\gamma_k,\gamma_k')\leq C$.
Say that $\overline{\sigma}'$ $C$-follows $\overline{\sigma}$ if for every $1\leq k\leq n$,
$d(\gamma_k,\gamma_k')\leq C$.

More generally, let $(x_1,...,x_n)$ and $(x_1',...,x_m')$ be two relative quasi-geodesic paths.
Say that they $C$-follow each other if for every $k$ with $k\leq n$ and $k\leq m$, we have
$d(x_k,x_k')\leq C$.

\begin{lemma}\label{Cfollowinggeodesics}
For every $c\geq 0$, there exists $C\geq 0$ such that the following holds.
Let $(x_1,...,x_n)$ and $(x_1',...,x_m')$ be two relative geodesics such that $x_1=x'_1$ and $d(x_n,x_m')\leq c$.
Then, $(x_1,...,x_n)$ and $(x_1',...,x_m')$ $C$-follow each other.
\end{lemma}

\begin{proof}
Roughly speaking,
the fact that the two relative geodesics relatively $C$-follow each other is a consequence of the Morse lemma in the hyperbolic space $(\hat{\Gamma},\hat{d})$.
The fact that they actually $C$-follow themselves for the Cayley graph distance is a consequence of the bounded coset penetration property.
More precisely, Lemma~\ref{Proposition315Osin} shows that there exists $C_0$ such that
$(x_1,...,x_n)$ and $(x_1',...,x_m')$ asynchronously fellow travel, that is for every $m$, there exists $l$ such that $d(x_m,x_l')\leq C_0$.
In particular, $m\leq l+C_0$ and similarly, $l\leq m+C_0$.

Assume first that $l<m$.
Let $y_j=x_{j-1}^{-1}x_j$ and similarly $y_j'=(x_{j-1}')^{-1}x_j'$.
We prove that there exists $C_1$ such that $d(e,y_{l+1}')\leq C_1$.
Indeed if $d(e,y_{l+1}')$ is large enough, then $y_{l+1}'$ lies in some parabolic subgroup $\mathcal{H}$
and by the BCP property, the relative geodesic $(x_1,...,x_n)$ passes through the coset $x_l'\mathcal{H}$ at two points $x_j$ and $x_{j+1}$, where
$d(x_l',x_j)$ and $d(x_{l+1}',x_{j+1})$ are bounded.
One necessarily has $j<m$.
Otherwise, the path obtained by adding to $x_1',...,x_l'$ an edge from $x_l'$ to $x_j$ would yield a path of length $l+1<j+1$ from $x_1'=x_1$ to $x_{j+1}$, contradicting the fact that $(x_1,...,x_{j+1})$ is a relative geodesic.
Thus, the relative geodesic defined by $(x_1,...,x_n)$ has to travel a long time inside $x_l'\mathcal{H}$ before it reaches $x_m$,
which proves that $d(x_m,x_l')$ is arbitrarily long if $d(e,y_{l+1}')$ is arbitrarily long, contradicting the fact that it is smaller than $C_0$, hence the existence of $C_1$.

If $l=m-1$, then $d(x_m,x'_m)\leq C_0+C_1$ and we are done.
Otherwise, we similarly prove that there exists $C_2$ such that $d(e,y_{l+2}')\leq C_2$ and so on, so that we can thus prove that for every $j\leq C_0$, as long as $l+j\neq m$
$d(e,y_j')$ is bounded.
We thus get that $d(x_l',x_m')\leq C_3$ for some $C_3$, so that
$d(x_m,x_m')\leq C_0+C_3$.

If $l>m$ we prove in the same way that $d(x_m,x_m')\leq C_0+C_3$.
This concludes the proof.
\end{proof}

In particular, we have the following.

\begin{lemma}\label{Cfollowingsequences}
There exists $C$ such that if $\overline{\sigma}=(\sigma_1,...,\sigma_n)$ and $\overline{\sigma}'=(\sigma_1',...,\sigma_n')$ are two reduced sequences mapped to the same element $\gamma$ in $\Gamma$,
then $\overline{\sigma}$ $C$-follows $\overline{\sigma}'$.
\end{lemma}

%We also prove an enhanced version of this lemma, for infinite relative geodesics.
%We will not need it in the proof of Theorem~\ref{thmcodingrelhypgroups}, but we will use it in Section~\ref{Sectionthermodynamicformalism}.
%We say that an infinite sequence $\overline{\sigma}=(\sigma_1,\sigma_2,...)$ is reduced if every finite subsequence is.

%\begin{lemma}\label{Cfollowinginfinitesequences}
%There exists $C$ such that the following holds.
%Let $\overline{\sigma}=(\sigma_1,\sigma_2,...)$ and $\overline{\sigma}'=(\sigma_1',\sigma_2',...)$ be two infinite reduced sequences such that $\gamma_n$ and $\gamma_n'$ converge to the same conical limit point in the Bowditch boundary,
%where, as usual, $\gamma_n=\sigma_1...\sigma_n$ and $\gamma_n'=\sigma_1'...\sigma_n'$.
%Then for every $n$, $d(\gamma_n,\gamma_n')\leq C$.
%\end{lemma}

%\begin{proof}
%We fix $n$.
%It follows from Lemma~\ref{Proposition315Osininfinite} that $d(\gamma_n,\gamma_m')\leq C$ for some $m$ and some $C$.
%The relative geodesics from $e$ to $\gamma_n$ and from $e$ to $\gamma_m$ do not end at the same point, but their ending points are within a uniform bounded distance of each other.
%We can thus still apply Lemma~\ref{Proposition315Osin}.
%There exists $C'$ such that the following hold.
%Let $k\leq n$.
%Then, there exists $k'$ such that $d(\gamma_k,\gamma_k')\leq C'$.
%Thus, we can use the same argument as in the proof of Lemma~\ref{Cfollowingsequences} to prove that $d(\gamma_k,\gamma_k')\leq C''$.
%\end{proof}

Similarly, we can prove the following lemma.
We will not use it to prove Theorem~\ref{thmcodingrelhypgroups}
but it will have useful consequences in the following, especially in the next paper.

\begin{lemma}\label{lemmarelativetripod}
Let $\overline{\sigma}=(\sigma_1,...,\sigma_n)$ and $\overline{\sigma}'=(\sigma_1',...,\sigma_m')$ be two reduced sequences.
Let $\overline{\gamma}$ and $\overline{\gamma}'$ be the corresponding sequences of elements in $\Gamma$.
Assume that the nearest point projection of $\gamma'_m$ on the geodesic $(e,\gamma_1,...,\gamma_n)$ in $\hat{\Gamma}$ is at $\gamma_l$.
If there are several such nearest point projection, choose the closest to $\gamma_n$.
Then, any relative geodesic from $\gamma_m'$ to $\gamma_n$ passes within a bounded distance (for the distance $d$) of $\gamma_{l}$.
Moreover, if $\gamma_l\neq e$, then any relative geodesic from $e$ to $\gamma_m'$ passes within a bounded distance of $\gamma_{l-1}$.
\end{lemma}

Before proving this lemma, let us give a brief heuristic explanation.
Typically, a relative geodesic from $e$ to $\gamma_m'$ will roughly follow a relative geodesic from $e$ to $\gamma_n$ up to some $\gamma_{l-1}$ and then both geodesics will diverge.
Assuming they diverge in the same parabolic subgroup, which is the worst possible case,
we let $\tilde{\gamma}_l$ be the exit point of the corresponding coset.
Then, using results of \cite{Sisto2}, we see that a relative geodesic from $\gamma_m'$ to $\gamma_n$ will first go to $\tilde{\gamma}_{l}$, then will travel in the same coset up to $\gamma_l$ and will finally roughly follow the relative geodesic from $\gamma_l$ to $\gamma_m$.
This is illustrated by the following picture.

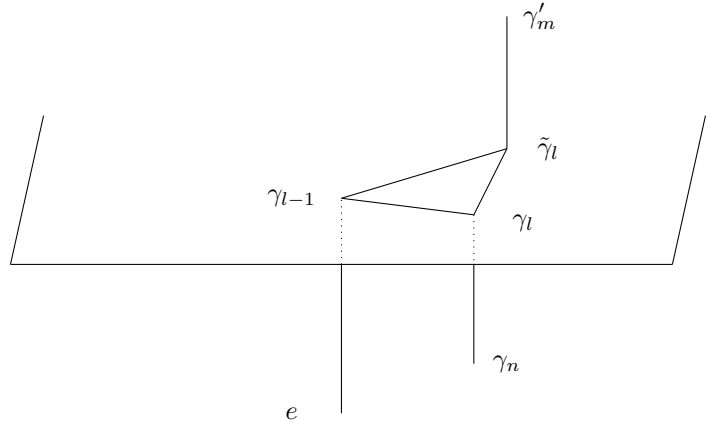
\begin{center}
\begin{tikzpicture}[scale=2.2]
\draw (-1.8,.5)--(-2,-.4);
\draw (2.2,.5)--(2,-.4);
\draw (-2,-.4)--(2,-.4);
\draw (0,-1.3)--(0,-.4);
\draw[dotted] (0,-.4)--(0,0);
\draw (0,0)--(.8,-.1);
\draw (0,0)--(1,.3);
\draw[dotted] (.8,-.1)--(.8,-.4);
\draw (.8,-.4)--(.8,-1);
\draw (1,.3)--(1,1.1);
\draw (.8,-.1)--(1,.3);
\draw (-.3,0)node{$\gamma_{l-1}$};
\draw (-.3,-1.3) node{$e$};
\draw (1.1,-.15) node{$\gamma_l$};
\draw (1.25,.3) node{$\tilde{\gamma}_l$};
\draw (1,-1) node{$\gamma_n$};
\draw (1.2,1.1) node{$\gamma'_m$};
\end{tikzpicture}
\end{center}

\begin{proof}
For simplicity, denote by $\gamma=\gamma_n$, $\gamma'=\gamma'_m$ and by $[e,\gamma]$ the relative geodesic defined by $\overline{\gamma}$ and similarly for $[e,\gamma']$.
Since $\gamma'$ projects on $[e,\gamma]$ at $\gamma_l$, any relative geodesic from $\gamma'$ to $\gamma$ roughly follows for the distance $\hat{d}$ a relative geodesic from $\gamma'$ to $\gamma_l$ and then a relative geodesic from $\gamma_l$ to $\gamma$ (this is true in any hyperbolic space, hence in $\hat{\Gamma}$, see for example \cite[Proposition~2.2]{MaherTiozzo}).
In particular, it passes within a bounded $\hat{d}$-distance of $\gamma_l$.
Moreover,
\begin{equation}\label{equationprojectiongeodesic}
    |\hat{d}(\gamma',\gamma)-[\hat{d}(\gamma',\gamma_l)+\hat{d}(\gamma_l,\gamma)]|\leq D_1,
\end{equation}
for some $D_1$.
Consider such a relative geodesic $[\gamma',\gamma]$.
Let $\alpha_1,...,\alpha_p$ be the consecutive points on $[\gamma',\gamma]$, with $\alpha_1=\gamma'$, and $\alpha_p=\gamma$.
We saw that there exists $i$ such that $\hat{d}(\alpha_i,\gamma_{l})\leq D_2$, for some $D_2$.
Consider a relative geodesic $[\alpha_i,\gamma_{l}]$ and denote by $\beta_1,...,\beta_q$ the consecutive points on this geodesic.
Then, since $\hat{d}(\alpha_i,\gamma_l)$ is bounded,~(\ref{equationprojectiongeodesic}) shows that concatenating $[\gamma',\alpha_i]$, $[\alpha_i,\gamma_{l}]$ and $[\gamma_{l},\gamma]$ yields a relative quasi-geodesic path from $\gamma'$ to $\gamma$ with bounded parameters. Denote this relative quasi geodesic by $\tilde{\alpha}$.

If $[\alpha_i,\gamma_{l}]$ does not enter any coset of a parabolic subgroup for more than some constant $D_3$ that we will choose later,then we also have $d(\alpha_i,\gamma_{l})\leq D_4$ for some $D_4$.
Otherwise, denote by $j$ the last time that $[\alpha_i,\gamma_{l}]$ enters such a coset, say $\tilde{\gamma}\mathcal{H}$, so that $d(\beta_{j+1},\gamma_{l})\leq D_4$.
Then, according to the bounded coset penetration property, the geodesic $[\gamma',\gamma]$ also enters this coset.
To deduce that the exit point is within a bounded distance of $\beta_{j+1}$, we need to show that $\tilde{\alpha}$ does not enter $\tilde{\gamma}\mathcal{H}$ after $\beta_{j+1}$.
Assume by contradiction this is the case.
Since $[\alpha_i,\gamma_l]$ is a relative geodesic, it only enters every coset at most once.
Therefore $[\gamma_l,\gamma]$ enters $\tilde{\gamma}\mathcal{H}$ and so there exsits $l'>l$ such that $\gamma_{l'}\in \tilde{\gamma}\mathcal{H}$.
Also, the concatenation of $[\gamma',\alpha_i]$ and $[\alpha_i,\gamma_{l}]$ is a relative quasi-geodesic from $\gamma'$ to $\gamma_l$, so if $D_3$ is chosen large enough, any relative geodesic from $\gamma'$ to $\gamma_l$ also needs to enter $\tilde{\gamma}\mathcal{H}$.
Replacing the edge in $\tilde{\gamma}\mathcal{H}$ of such a relative geodesic by an edge with the same origin and ending at $\gamma_{l'}$, we get a path in $\hat{\Gamma}$ satisfying that $\hat{d}(\gamma',\gamma_{l'})\leq \hat{d}(\gamma',\gamma_l)$.
This is a contradiction, since $\gamma_l$ is the projection of $\gamma'$ onto $[e,\gamma]$ which is the closest possible to $\gamma$.
Thus, $\tilde{\alpha}$ does not enter $\tilde{\gamma}\mathcal{H}$ after $\beta_{j+1}$ and according to the BCP property,
we can find $\alpha_{i'}$ such that $d(\alpha_{i'},\beta_{j+1})\leq D_5$.
Finally, $d(\alpha_{i'},\gamma_{l})\leq D_6$, which concludes the first part of the lemma.

We prove that a relative geodesic $[e,\gamma']$ from $e$ to $\gamma'$ passes within a bounded distance of $\gamma_{l-1}$ in the same way.
We know by the same reasoning as above that such a relative geodesic passes within a bounded $\hat{d}$-distance of $\gamma_l$, but we cannot guaranty that $\gamma_l$ is actually within a bounded distance of the exit point of $[e,\gamma']$ in the corresponding coset, but we necessarily have that $\gamma_{l-1}$ is a bounded distance from the entering point.
\end{proof}

Combining this with Lemma~\ref{Cfollowinggeodesics}, we get the following result.
Again, we will not use it to prove Theorem~\ref{thmcodingrelhypgroups}, but it will be useful in the following, especially in our next paper.

\begin{lemma}\label{fellowtravelling}
There exists $c\geq 0$ and $C\geq 0$ such that the following holds.
Let $(\gamma_1,...,\gamma_n)$ and $(\gamma_1',...,\gamma_m')$ be two relative geodesics, with $\gamma_n=\gamma_m'$.
Assume that the nearest point projection of $\gamma'_1$ on the geodesic $(\gamma_1,...,\gamma_n)$ in $\hat{\Gamma}$ is at $\gamma_l$.
If there are several such nearest point projection, choose the closest to $\gamma_n$.
Then, there exists $k$ and $j$, with $|k-l|\leq c$ such that for every $i$ with $j+i\leq m$ and $k+i\leq n$, we have
$d(\gamma_{k+i},\gamma_{j+i})\leq C$.
\end{lemma}

In other words, the two relative geodesics begin $C$-following each other at some point.
The moment they begin doing so is approximately the moment that the second geodesic reaches its projection on the first one.

\begin{proof}
According to Lemma~\ref{lemmarelativetripod}, the relative geodesic $(\gamma_1',...,\gamma_n')$ passes within a bounded distance of $\gamma_l$,
so there exists $p$ such that $d(\gamma_l,\gamma_p')\leq C_0$.
Consider the reversed geodesics $(\gamma_n,\gamma_{n-1},...,\gamma_l)$ and $(\gamma_m',\gamma_{m-1}',...,\gamma_p')$.
They begin at the same point so we can apply Lemma~\ref{Cfollowinggeodesics} to conclude.
\end{proof}

Recall that $B_C$ is the ball of center $e$ and of radius $C$ in the Cayley graph of $\Gamma$ with respect to $S$.
Denote by $P_C$ the (finite) set of subsets of $B_C\setminus \{e\}$ and define $V_1=V_0\times P_C$ (where $V_0$ is the set of vertices of $\mathcal{G}_0$).

We define a graph $\mathcal{G}_1$ with set of vertices $V_1$ as follows.
First consider a nicely reduced sequence $\overline{\sigma}=(\sigma_1,...,\sigma_n)\in \mathcal{N}$.
For all $1\leq k \leq n$, denote by $P_k$ the set of elements $\gamma\in B_C$ such that
one can find a sequence $(\sigma_1',...,\sigma_k')$ satisfying that
\begin{itemize}
    \item $(\sigma_1',...,\sigma_k')$ is strictly smaller than $(\sigma_1,...,\sigma_k)$ for the lexicographical order,
    \item $(\sigma_1',...,\sigma_k')$ $C$-follows  $(\sigma_1,...,\sigma_k)$,
    \item  $\sigma_1'...\sigma_k'=\sigma_1...\sigma_k\gamma$ in $\Gamma$.
\end{itemize}
Since $\overline{\sigma}$ is nicely reduced, $e\notin P_k$.
By definition, if $P_0:=\emptyset$, then
\begin{align*}
    P_k&=\{\gamma\in B_C, \text{ such that }\exists \sigma \in \Sigma \text{ strictly smaller than } \sigma_k \text{ with }\gamma=\sigma_k^{-1}\sigma\}\\
    &\cup \{\gamma\in B_C, \text{ such that }\exists \sigma \in \Sigma \text{ and } \gamma'\in P_{k-1}\text{ with }\gamma=\sigma_k^{-1}\gamma' \sigma\}.
\end{align*}

The first set corresponds to the case where $(\sigma_1',...,\sigma_{k-1}')=(\sigma_1,...,\sigma_{k-1})$ and the second one corresponds to the case where
$(\sigma_1',...,\sigma_{k-1}')$ is smaller than $(\sigma_1,...,\sigma_{k-1})$.

\medskip
We can now define the edges of $\mathcal{G}_1$.
Let $w=(v,P)$ and $w'=(v',P')$ be two vertices of $V_1$, with $v,v'\in V_0$ and $P,P'\in P_C$.
Define an edge from $w$ to $w'$ if the following two conditions are satisfied.
First, there is an edge in $\mathcal{G}_0$ from $v$ to $v'$, with label $\sigma$.
Second,
\begin{align*}
    P'&=\{\gamma\in B_C, \text{ such that }\exists \sigma' \in \Sigma \text{ smaller than } \sigma \text{ with }\gamma=\sigma^{-1}\sigma'\}\\
    &\cup \{\gamma\in B_C, \text{ such that }\exists \sigma' \in \Sigma \text{ and } \gamma'\in P\text{ with }\gamma=\sigma^{-1}\gamma' \sigma'\}.
\end{align*}
Then, label this edge from $w$ to $w'$ by $\sigma$ too.
Denote by $\phi$ the induced labelling map.

\begin{proposition}\label{nicelyreducedlanguage}
Recall that $v_0$ is the distinguished vertex of $\mathcal{G}_0$.
Assume that $(v_0,\emptyset),(v_1,P_1),...,(v_n,P_n)$ is a sequence of adjacent vertices in $\mathcal{G}_1$, starting at $(v_0,\emptyset)$, with edges $e_1,...,e_n$ labelled with $\sigma_1,...,\sigma_n$.
Then, $\overline{\sigma}=(\sigma_1,...,\sigma_n)$ is nicely reduced.
\end{proposition}

\begin{proof}
Assume it is not nicely reduced.
Then, let $\overline{\sigma}'=(\sigma_1',...,\sigma_n')$ be the nicely reduced sequence with $\sigma_1...\sigma_n=\sigma_1'...\sigma_n'$.
Denote by $i$ the first time $\sigma_i\neq \sigma_i'$, so that $\sigma_i'<\sigma_i$.
According to Lemma~\ref{Cfollowingsequences}, we have $d(\sigma_1...\sigma_i,\sigma_1'...\sigma_i')\leq C$ and since $\sigma_1=\sigma_1'$,...,$\sigma_{i-1}=\sigma_{i-1}'$, we have $d(\sigma_i,\sigma_i')\leq C$.
Thus, $\sigma_i^{-1}\sigma_i'\in P_i$.
This in turn proves that $\sigma_{i+1}^{-1}\sigma_i^{-1}\sigma_i'\sigma_{i+1}'\in P_{i+1}$,...,$\sigma_{n}^{-1}...\sigma_i^{-1}\sigma_i'...\sigma_{n}'\in P_n$.
However, we have $\sigma_1...\sigma_{i-1}=\sigma_1'...\sigma_{i-1}'$ and $\sigma_1...\sigma_{n}=\sigma_1'...\sigma_{n}'$, so that $e\in P_n$, which is a contradiction.
\end{proof}

Finally, we choose the distinguished vertex $v_*=(v_0,\emptyset)$ and denote by $\mathcal{G}$ the subgraph of $\mathcal{G}_1$ whose set of vertices $V$ consists of those one can reach from $v_*$.
Since we chose the same labelling maps for $\mathcal{G}$ and $\mathcal{G}_0$, 
the graph $\mathcal{G}$ still satisfies the three first conditions of Definition~\ref{definitionautomaticstructure}.
Also, Proposition~\ref{nicelyreducedlanguage} shows that the sequences one can construct with the labelling map are necessarily nicely reduced, so that the fourth condition of Definition~\ref{definitionautomaticstructure} also is satisfied.
This proves Theorem~\ref{thmcodingrelhypgroups}. \qed

\section{Green functions of the parabolic subgroups}\label{SectionGreenparabolics}
We consider a group $\Gamma$, hyperbolic relative to a collection of peripheral subgroups $\Omega$ and we fix a finite collection $\Omega_0=\{\mathcal{H}_1,...,\mathcal{H}_N\}$ of representatives of conjugacy classes of $\Omega$.
We assume that $\Gamma$ is non-elementary.
Let $\mu$ be a probability measure on $\Gamma$, $R_{\mu}$ the spectral radius of the $\mu$-random walk and $G(\gamma,\gamma'|r)$ the associated Green function, evaluated at $r$, for $r\in [0,R_{\mu}]$.
If $\gamma=\gamma'$, we simply use the notation $G(r)=G(\gamma,\gamma|r)=G(e,e|r)$.

\medskip
Recall the following notations from Section~\ref{Sectionspectralpositiverecurrence}.
We let $p_{k,r}$ be the transition kernel of first return to $\mathcal{H}_k$, associated with $r\mu$.
Also, we let $G_{k,r}$ be the associated Green function and $R_k(r)$ be the associated spectral radius.
We write $R_k=R_k(R_\mu)$.

Our goal in this section is to collect some properties of $G_{k,r}$ and $R_{k,r}$ and to relate them with properties of the initial Green function $G$.
We will refer to $G_{k,r}$ and $R_{k,r}$ as the parabolic Green functions and the parabolic spectral radii.

\subsection{Derivatives of the parabolic Green functions}\label{SectionderivativesparabolicGreen}
We begin this section with some geometric lemmas.
We fix a generating set $S$.
The following is a direct consequence of \cite[Lemma~2.10]{Sisto2}.

\begin{lemma}\label{lemmaboundedprojection}
Let $\Gamma$ be a relatively hyperbolic group.
There exists $C\geq 0$ such that the following holds.
Let $\xi$ be a parabolic limit point and let $\mathcal{H}$ be its stabilizer.
Then, there exists a neighborhood $\mathcal{U}$ of $\xi$ in the Bowditch compactification $\Gamma \cup \partial_B\Gamma$ such that the nearest point projection of $\Gamma \cap \mathcal{U}^c$ on $\mathcal{H}$ has diameter at most $C$.
\end{lemma}

We deduce the following.

\begin{lemma}\label{lemmafinitesetloxodromics}
Let $\Gamma$ be a non-elementary relatively hyperbolic group and let $\mathcal{H}$ be a maximal parabolic subgroup.
There exists $C\geq0$ such that the following holds.
There exists a finite set $F=\{\gamma_1,\gamma_2,\gamma_3\}$ of loxodromic elements such that for any $\gamma,\gamma'\in \Gamma$, there exists $i$ such that
the nearest points projections of $\gamma_i\gamma$ and $\gamma_i\gamma'$ on $\mathcal{H}$ are at distance at most $C$ from $e$.
\end{lemma}

\begin{proof}
Let $\xi$ be the parabolic fixed point of $\mathcal{H}$.
Let $\mathcal{U}$ be the neighborhood of $\xi$ in the Bowditch compactification given by Lemma~\ref{lemmaboundedprojection}.
Since $\Gamma$ is non-elementary, we can find three loxodromic elements $\gamma_1$, $\gamma_2$ and $\gamma_3$ whose attrative fixed points $\gamma_1^+,\gamma_2^+,\gamma_3^+$ are all dinstinct and lie in $\mathcal{U}^c$.
Fix a neighborhood $\mathcal{U}_i$ of $\gamma_i^+$ contained in $\mathcal{U}^c$.
Also, fix three disjoint neighborhoods $\mathcal{V}_i$ of the repelling fixed points $\gamma_i^-$ of $\gamma_i$.
Then, up to changing $\gamma_i$ by some power of itself, the image of the complement of $\mathcal{V}_i$ is contained in $\mathcal{U}_i$.

Since $\mathcal{V}_i$ are all disjoint, there exists $i$ such that both $\gamma\notin \mathcal{V}_i$ and $\gamma'\notin \mathcal{V}_i$.
In particular, $\gamma_i\gamma$ and $\gamma_i\gamma'$ both lie in the complement of $\mathcal{U}$.
Hence, their projection on $\mathcal{H}$ are within a uniform bounded distance of $e$, according to Lemma~\ref{lemmaboundedprojection}.
\end{proof}

We use these lemmas to prove the following generalization of \cite[Lemma~1.13]{Yang1} (see also \cite[Lemma~2.4]{Gouezel1} for a similar statement for hyperbolic groups).

\begin{lemma}\label{lemmaYangGouezel}
There exists $C\geq 0$ and $c\geq 0$ such that for any $\gamma,\gamma'\in \Gamma$, one can find some point $\sigma\in B(e,C)$ such that $\hat{d}(e,\gamma \sigma \gamma')\geq \hat{d}(e,\gamma)+\hat{d}(e,\gamma')-c$.
Morever, if $\alpha$ is a relative geodesic from $e$ to $\gamma \sigma \gamma'$, then $\gamma$ is within a uniform distance of a point on $\alpha$.
\end{lemma}

\begin{proof}
First, we can assume that there exists an infinite maximal parabolic subgroup.
Otherwise, $\Gamma$ is hyperbolic and the result is given by \cite[Lemma~2.4]{Gouezel1}.
Fix such a parabolic subgroup $\mathcal{H}$.
Choose a set $F=\{\gamma_1,\gamma_2,\gamma_3\}$ as in Lemma~\ref{lemmafinitesetloxodromics}.

Fix $\gamma,\gamma'\in \Gamma$.
According to Lemma~\ref{lemmafinitesetloxodromics}, there exists $i$ such that
the projection of $\gamma^{-1}$ on $\gamma_i^{-1}\mathcal{H}$ is within a uniformly bounded distance of $\gamma_i^{-1}$ and the projection of $\gamma_i\gamma'$ on $\mathcal{H}$ is within a uniformly bounded distance of $e$.

The BCP property shows that the projection of $\gamma^{-1}$ on $\gamma_i^{-1}\mathcal{H}$ in $\hat{\Gamma}$ also is within a uniformly $d$-bounded distance of $\gamma_i^{-1}$.
More precisely, we have the following.
Let $\beta_0$ be a relative geodesic from $\gamma^{-1}$ to $\gamma_i^{-1}\mathcal{H}$ and denote its endpoint by $\gamma_i^{-1}\sigma_0$.
Then, \cite[Lemma~1.13~(2)]{Sisto2} shows that $d(e,\sigma_0)$ is uniformly bounded.
Similarly, let $\beta_1$ be a relative geodesic from $\mathcal{H}$ to $\gamma_i\gamma'$ and denote its starting point by $\sigma_1$.
Then, $d(e,\sigma_0),d(e,\sigma_1)\leq C_1$ for some uniform $C_1$.
Letting $\sigma_2\in \mathcal{H}$, consider the translated path $\gamma_i^{-1}\sigma_2\beta_1$ which starts at $\gamma_i^{-1}\sigma_2 \sigma_1$ and ends at $\gamma_i^{-1}\sigma_2 \gamma_i\gamma'$.
Finally, let $\beta_2$ be the concatenation of $\beta_0$, a path on length 1 in $\hat{\Gamma}$ from $\gamma_i^{-1}\sigma_0$ to $\gamma_i^{-1}\sigma_2 \sigma_1$ and the translated path $\gamma_i^{-1}\sigma_2\beta_1$.
Then, $\beta_2$ starts at $\gamma^{-1}$, passes through $\gamma_i^{-1}\sigma_0$ and ends at $\gamma_i^{-1}\sigma_2 \gamma_i\gamma'$.

Then, $\beta_2$ is a 2-local relative geodesic.
Thus, if $d(e,\sigma_2)$ is large enough, independently of $\gamma$ and $\gamma'$, Lemma~\ref{Theorem52AntolinCiobanu} shows that $\beta_2$ is a $(\lambda,c)$-quasi geodesic, for uniform $\lambda$ and $c$.
Thus, the BCP property shows that a relative geodesic from $\gamma^{-1}$ to $\gamma_i^{-1}\sigma_2\gamma_i\gamma'$ passes within a bounded distance of $\gamma_i^{-1}\sigma_0$.

Fix such a $\sigma_2$ large enough and let $\sigma=\gamma_i^{-1}\sigma_2\gamma_i$.
Finally, let $\alpha$ be a relative geodesic from $e$ to $\gamma\sigma \gamma'$.
Then, $\alpha$ passes within a uniformly bounded distance of $\gamma\gamma_i^{-1}\sigma_0$.
Since $d(e,\sigma_0)\leq c$ and since $\gamma_i^{-1}$ is fixed, it also passes within a uniformly bounded distance of $\gamma$.
Moreover, since $\sigma_2$ is fixed, there is a finite number of possibilities for $\sigma$, so that $d(e,\sigma)\leq C$ for some uniform $C$.
This concludes the proof.
\end{proof}

We first deduce from this the following lemma, adapted from \cite[Lemma~2.5]{Gouezel1}.
Denote by $\hat{S}_m$ the relative sphere of radius $m$, that is,
elements $\gamma\in \Gamma$ such that $\hat{d}(e,\gamma)=m$, where $\hat{d}$ is the distance in $\hat{\Gamma}$.
For simplicity, for $r\in [0,R_{\mu}]$, we write
$H(e,\gamma|r)=G(e,\gamma|r)G(\gamma,e|r)$.

\begin{lemma}\label{finitesumalongspheres}
There exists some uniform $C\geq 0$ such that for every $r\in [0,R_{\mu}]$, for every $m$,
$$\sum_{\gamma\in \hat{S}_m}H(e,\gamma|r)\leq C.$$
\end{lemma}

\begin{proof}
Since $H(\gamma,\gamma'|r)\leq H(\gamma,\gamma'|r')$ if $r\leq r'$, it suffices to prove the proposition for $r=R_{\mu}$.
Fix $r<R_{\mu}$
and write $u_m(r)=\sum_{\gamma\in \hat{S}_m}H(e,\gamma|r)$.
For $\gamma\in \hat{S}_m$, $\gamma'\in \hat{S}_{m'}$, Lemma~\ref{lemmaYangGouezel} shows that one can define a point
$\gamma \sigma \gamma'\in \underset{m+m'\leq l\leq m+m'+m_0}{\bigcup}\hat{S}_{l}$, where $m_0$ is fixed and where $d(e,\sigma)$ is bounded.
Moreover, $\gamma$ is a within a uniformly bounded distance of a relative geodesic from $e$ to $\gamma \sigma \gamma'$.
Choosing such a $\sigma$, we write $\Psi(\gamma,\gamma')=\gamma \sigma \gamma'$.
Since $d(e,\sigma)$ is bounded, we have
$$H(e,\gamma|r)H(e,\gamma'|r)\leq C_0 H(e,\gamma|r)H(e,\sigma \gamma'|r)=C_0H(e,\gamma|r)H(\gamma,\gamma \sigma \gamma'|r).$$
Also, $H(e,\gamma|r)H(\gamma,\gamma \sigma \gamma'|r)\leq C_0'H(e,\gamma\sigma\gamma')$ (see~(\ref{triangleGreen}), so that
$$H(e,\gamma|r)H(e,\gamma'|r)\leq CH(e,\gamma \sigma \gamma'|r).$$

Fix $\gamma''\in \underset{m+m'\leq l\leq m+m'+m_0}{\bigcup}\hat{S}_{l}$.
Assume that $\gamma''=\gamma \sigma \gamma'$, with $\gamma \in \hat{S}_m$ within a bounded distance of a point on a relative geodesic from $e$ to $\gamma''$ and $\sigma\in B(e,C)$.
Fix such a relative geodesic $[e,\gamma'']$ from $e$ to $\gamma''$.
There exists $\tilde{\gamma}$ on $[e,\gamma'']$ such that $d(\gamma,\tilde{\gamma})\leq C$.
Thus, we also have $\hat{d}(\gamma,\tilde{\gamma})\leq C$.
In particular, $|\hat{d}(e,\tilde{\gamma})-m|\leq C$, so there is a uniformly bounded number of distinct possibilities for $\tilde{\gamma}$.
Since $d(\gamma,\tilde{\gamma})$ is bounded, this gives a uniformly bounded number of possibilites for $\gamma$.
Also, since $d(e,\sigma)$ is bounded, there is a uniformly bounded number of possibilites for $\sigma$.
In particular, there is a uniformly bounded number of such decomposition of $\gamma''$.
In other words, the map $\Psi$ has a uniformly bounded number of preimages of $\gamma''$.

This proves that
\begin{align*}
    u_m(r)u_{m'}(r)&=\sum_{\gamma\in \hat{S}_m,\gamma'\in \hat{S}_{m'}}H(e,\gamma|r)H(e,\gamma'|r)\leq C\sum_{\gamma\in \hat{S}_m,\gamma'\in \hat{S}_{m'}}H(e,\Psi(\gamma,\gamma')|r)\\
    &\leq C'\left (\sum_{\gamma''\in \hat{S}_{m+m'}}H(e,\gamma''|r)+\cdots +\sum_{\gamma''\in \hat{S}_{m+m'+m_0}}H(e,\gamma''|r)\right )
\end{align*}
so that
\begin{equation}\label{Lemma2.5Gouezel}
    u_m(r)u_{m'}(r)\leq C'\sum_{j=0}^{m_0} u_{m+m'+j}(r).
\end{equation}

Since $r<R_{\mu}$, $\frac{d}{dr}G(e,\gamma|r)$ is finite, so that according to Lemma~\ref{lemmafirstderivative}, the sum $\sum_{\gamma\in \Gamma}H(e,\gamma|r)$ is finite.
Hence, the sequence $u_m(r)$ converges to 0 when $m$ tends to infinity, so that it reaches its maximum at some index $k_0(r)$.
This index does depend on $r$. However, denoting by $M(r)$ this maximum, Equation~(\ref{Lemma2.5Gouezel}) shows that
$M(r)^2\leq (m_0+1)C'M(r)$, so that $M(r)\leq (m_0+1)C'$.
This constant $(m_0+1)C'$ does not depend on $r$, so that we also have
$u_m(R_{\mu})\leq (m_0+1)C'$ for every $m$, which is the desired inequality.
\end{proof}

Applying Lemma~\ref{finitesumalongspheres} for $m=1$, we get the following corollary, that will be very useful.
\begin{corollary}\label{derivativeparabolicGreenfinite}
For every parabolic subgroup $\mathcal{H}\in \Omega_0$ and every $r\in [0,R_{\mu}]$, we have
$$\sum_{h\in \mathcal{H}}G(e,h|r)G(h,e|r)<+\infty.$$
\end{corollary}

\subsection{Relations between $G$, $G_k$ and $R_k$}\label{SectionrelationsdifferentGreens}
We prove here Theorem~\ref{maintheoremGreen}.
Recall that
$$I^{(j)}(r)=\sum_{\gamma_1,...,\gamma_j\in \Gamma}G(e,\gamma_1|r)G(\gamma_1,\gamma_2|r)...G(\gamma_{j-1},\gamma_j|r)G(\gamma_j,e|r).$$
Also define, for a parabolic subgroup $\mathcal{H}_k\in \Omega_0$,
$$I_{k}^{(j)}(r)=\sum_{\gamma_1,...,\gamma_j\in \mathcal{H}_k}G_{k,r}(e,\gamma_1|1)G_{k,r}(\gamma_1,\gamma_2|1)...G_{k,r}(\gamma_{j-1},\gamma_j|1)G_{k,r}(\gamma_j,e|1).$$
Finally, for simplicity, denote by
$G^{(j)}_{k,r}$ the $j$th derivative of the parabolic Green functions at 1, that is,
$G^{(j)}_{k,r}=\frac{d^j}{dt^j}_{|t=1}\left (G_{k,r}(e,e|t)\right )$ and by $G_{r}^{(j)}$ the $j$th derivative of the Green function on the whole group at $r$.

Beware that $G^{(j)}_{k,r}$ is a derivative of $t\mapsto G_{k,r}(e,e|t)$ and not of $r\mapsto G_{k,r}(e,e|t)$, whereas $G_{r}^{(j)}$ is a derivative of $r\mapsto G(e,e|r)$.
These notations are far from being perfect but they will be very convenient in the following.
We prove the main result of this section, which states the a priori estimates we will need to prove the local limit theorem.

\begin{proposition}\label{roughequadiff}
For every $r\in [0,R_{\mu})$, we have
$$\frac{G^{(2)}_r}{\left (G^{(1)}_r\right )^3}\asymp 1+\sum_{k}G^{(2)}_{k,r}.$$
\end{proposition}

Using the relations between $G_{k,r}^{(j)}$ and $I_{k}^{(j)}$ established in the Section~\ref{SectionCombinatorialGreen}, we can restate the statement of the proposition as
\begin{equation}\label{roughequadiffI_2}
    \frac{I^{(2)}(r)}{I^{(1)}(r)^3}\asymp 1+ \sum_k I^{(2)}_{k}(r).
\end{equation}

In particular, Theorem~\ref{maintheoremGreen} is an immediate consequence of Proposition~\ref{roughequadiff}.
The proof of~(\ref{roughequadiffI_2}) is very technical and we first explain roughly the arguments.
Recall that
$$I^{(2)}(r)=\sum_{\gamma,\gamma'}G(e,\gamma|r)G(\gamma,\gamma'|r)G(\gamma',e|r).$$
We fix two elements $\gamma$ and $\gamma'$ and consider the moment relative geodesics from $e$ to $\gamma$ and from $e$ to $\gamma'$ start diverging.
Typically, two such relative geodesics will roughly follow each other up to some $\gamma_0$ and then diverge inside parabolic subgroups.
Then to go from $e$ to $\gamma$, the random walk passes near $\gamma_0$, according to weak relative Ancona inequalities.
Similarly, to go from $\gamma'$ to $e$, it will pass near $\gamma_0$.

If the two relative geodesic diverge into different parabolic subgroups, then to go from $\gamma$ to $\gamma'$, the random walk also has to pass near $\gamma_0$, as suggested by the picture below.
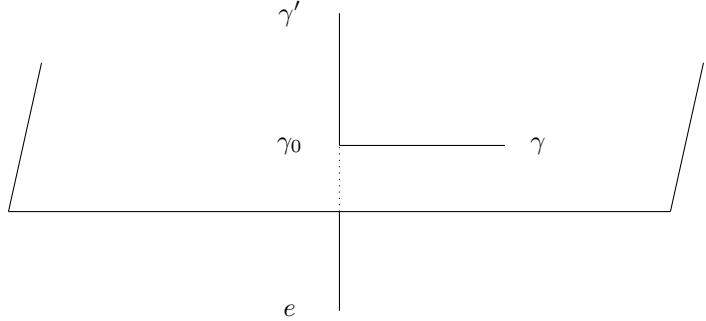
\begin{center}
\begin{tikzpicture}[scale=2.2]
\draw (-1.8,.5)--(-2,-.4);
\draw (2.2,.5)--(2,-.4);
\draw (-2,-.4)--(2,-.4);
\draw (0,0)--(0,.8);
\draw (0,-1)--(0,-.4);
\draw[dotted] (0,-.4)--(0,0);
\draw (0,0)--(1,0);
\draw (-.3,0)node{$\gamma_0$};
\draw (-.3,-1) node{$e$};
\draw (1.2,0) node{$\gamma$};
\draw (-.3,.8) node{$\gamma'$};
\end{tikzpicture}
\end{center}
Summing over these $\gamma,\gamma'$, we thus transform the sum in $I^{(2)}$ into
$$\asymp \sum_{\gamma,\gamma'}G(e,\gamma_0)G(\gamma_0,e)G(\gamma_0,\gamma)G(\gamma,\gamma_0)G(\gamma_0,\gamma')G(\gamma',\gamma_0).$$
Changing the sum over $\gamma$ and $\gamma'$ with a sum over $\gamma_0$, $\gamma_0^{-1}\gamma$ and $\gamma_0^{-1}\gamma'$, we obtain a sum $\Sigma \asymp (I^{(1)})^3$.

\medskip
On the contrary, if the two relative geodesic diverge into the same parabolic subgroups, then we obtain a triangle inside the corresponding parabolic subgroup $\mathcal{H}=\mathcal{H}_k$, as illustrated by the picture below.
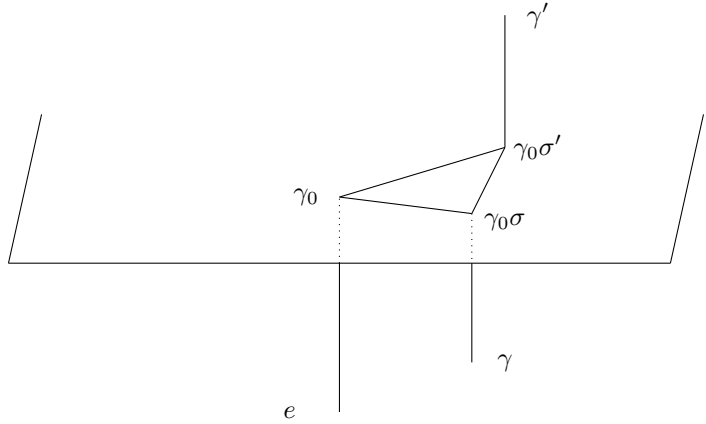
\begin{center}
\begin{tikzpicture}[scale=2.2]
\draw (-1.8,.5)--(-2,-.4);
\draw (2.2,.5)--(2,-.4);
\draw (-2,-.4)--(2,-.4);
\draw (0,-1.3)--(0,-.4);
\draw[dotted] (0,-.4)--(0,0);
\draw (0,0)--(.8,-.1);
\draw (0,0)--(1,.3);
\draw[dotted] (.8,-.1)--(.8,-.4);
\draw (.8,-.4)--(.8,-1);
\draw (1,.3)--(1,1.1);
\draw (.8,-.1)--(1,.3);
\draw (-.2,0)node{$\gamma_0$};
\draw (-.3,-1.3) node{$e$};
\draw (1,-.15) node{$\gamma_0\sigma$};
\draw (1.2,.3) node{$\gamma_0\sigma'$};
\draw (1,-1) node{$\gamma$};
\draw (1.2,1.1) node{$\gamma'$};
\end{tikzpicture}
\end{center}
This time, to go from $\gamma$ to $\gamma'$, the random walk does not have to pass near $\gamma_0$, but it has to pass near $\gamma_0\sigma$ and $\gamma_0\sigma'$, the respective projections of $\gamma$ and $\gamma'$ on $\gamma_0\mathcal{H}_k$.
Then, summing over these $\gamma,\gamma'$ we get
\begin{align*}
    \asymp \sum_{\gamma,\gamma'}\sum_{\sigma ,\sigma'\in H}G(e,\gamma_0)&G(\gamma_0,e)G(e,\sigma)G(\sigma,\sigma')G(\sigma',e)\\
    &G(\gamma_0\sigma,\gamma)G(\gamma,\gamma_0\sigma)G(\gamma_0\sigma',\gamma')G(\gamma',\gamma_0\sigma').
\end{align*}
Changing this last sum with a sum over $\gamma_0$, $\sigma$, $\sigma'$, $\sigma^{-1}\gamma_0^{-1}\gamma$ and $(\sigma')^{-1}\gamma_0^{-1}\gamma'$, we obtain a sum $\Sigma\asymp (I^{(1)})^3I_k^{(2)}$.
Summing over all possible $\mathcal{H}_k$, we get~(\ref{roughequadiffI_2}).

\begin{proof}
Let us now give the formal proof.
Recall that henever $f$ and $g$ are two functions satisfying that there exists $C$ such that $f\leq Cg$, we write
$f\lesssim g$.
In order to prove~(\ref{roughequadiffI_2}), we need to prove an upper bound and a lower bound.
We start proving the upper bound, that is,
\begin{equation}\label{roughequadiffupperbound}
     \frac{I^{(2)}(r)}{I^{(1)}(r)^3}\lesssim 1+ \sum_k I_k^{(2)}(r).
\end{equation}
Consider, for each $\gamma\in \Gamma$, a relative geodesic $[e,\gamma]$ from $e$ to $\gamma$.
Also assume that if $\gamma'\in [e,\gamma]$, then the chosen relative geodesic $[e,\gamma']$ coincides with the restriction of $[e,\gamma]$ from $e$ to $\gamma'$.
This is possible using for example the automaton $\mathcal{G}$ given by Theorem~\ref{thmcodingrelhypgroups}.
We want to control the sum
$$I^{(2)}(r)=\sum_{\gamma,\gamma'}G(e,\gamma|r)G(\gamma,\gamma'|r)G(\gamma',e|r).$$
If $\gamma$ is fixed, we consider for every $\gamma'$ the nearest point projection of $\gamma'$ on $[e,\gamma]$.
If there are two possible such projections, we choose the closest to $e$.
For $\gamma_0\in [e,\gamma]$, we denote by $\Gamma_{\gamma_0}(\gamma)$ the set of $\gamma'$ such that this nearest point projection is at $\gamma_0$.
Also, when $\gamma_0$ is fixed, we denote by $\Gamma^{\gamma_0}$ the set of $\gamma$ such that $\gamma_0\in [e,\gamma]$.
We can bound $I^{(2)}$ from above by
$$\sum_{\gamma_0\in \Gamma}\sum_{\gamma\in \Gamma^{\gamma_0}}\sum_{\gamma'\in \Gamma_{\gamma_0}(\gamma)}G(e,\gamma|r)G(\gamma,\gamma'|r)G(\gamma',e|r).$$
According to Lemma~\ref{lemmarelativetripod}, a relative geodesic from $\gamma'$ to $e$ passes within a bounded distance of $\gamma_0$.
Weak relative Ancona inequalities then show that
$$G(e,\gamma|r)\asymp G(e,\gamma_0|r)G(\gamma_0,\gamma|r)\text{ and }G(\gamma',e|r)\asymp G(\gamma',\gamma_0|r)G(\gamma_0,e|r).$$
We thus get
\begin{equation}\label{proofroughequadiff-1}
\begin{split}
    I^{(2)}(r)\lesssim \sum_{\gamma_0\in \Gamma}\sum_{\gamma\in \Gamma^{\gamma_0}}\sum_{\gamma'\in \Gamma_{\gamma_0}(\gamma)}G(e,\gamma_0|r)&G(\gamma_0,\gamma|r)\\
    &G(\gamma,\gamma'|r)G(\gamma',\gamma_0|r)G(\gamma_0,e|r).
\end{split}
\end{equation}

\medskip
We fix $\gamma_0$.
We claim that
\begin{equation}\label{proofroughequadiff0}
    \sum_{\gamma\in \Gamma^{\gamma_0}}\sum_{\gamma'\in \Gamma_{\gamma_0}(\gamma)}G(\gamma_0,\gamma|r)G(\gamma,\gamma'|r)G(\gamma',\gamma_0|r)\lesssim I^{(1)}(r)^2(1+\sum_{k}I_k^{(2)}(r)).
\end{equation}

Our goal is to prove this claim.
By translating by $\gamma_0$ the relative geodesics $[e,\gamma_0^{-1}\gamma]$ and $[e,\gamma_0^{-1}\gamma']$, we get relative geodesics $[\gamma_0,\gamma]$ and $[\gamma_0,\gamma']$ from $\gamma_0$ to $\gamma$ and from $\gamma_0$ to $\gamma'$ respectively.
We distinguish the different elements $\gamma$ in $\Gamma^{\gamma_0}$ according to the first jump of $[\gamma_0,\gamma]$.
This first jump can be in some parabolic subgroup $\mathcal{H}_k$ or in the generating set $S$.
Fors simplicity, we write $S=\mathcal{H}_0$ in all this proof.
We denote by $\Gamma^{\gamma_0}_k$ the subset of $\Gamma^{\gamma_0}$ consisting of elements $\gamma$ such that the first jump is in $\mathcal{H}_k$, $0\leq k \leq N$.
Notice that we do not ask that the geodesic from $e$ to $\gamma$ satisfies that the first jump after $\gamma_0$ is in $\mathcal{H}_k$.
We really ask that the translated geodesic from $\gamma_0$ to $\gamma$ starts with a jump in $\mathcal{H}_k$.
We do the same for $\gamma'\in \Gamma_{\gamma_0}(\gamma)$ and denote by $\Gamma_{\gamma_0}^j(\gamma)$ the corresponding set of $\gamma'$ such that the first jump of $[\gamma_0,\gamma']$ is in $\mathcal{H}_j$.
For $\gamma\in \Gamma^{\gamma_0}_k$, we denote by $\sigma \in \mathcal{H}_k$ the first jump of $[\gamma_0,\gamma]$.
Similarly for $\gamma'\in \Gamma_{\gamma_0}^j(\gamma)$, we denote by $\sigma'\in \mathcal{H}_j$ the first jump of $[\gamma_0,\gamma']$.

\medskip
We prove that any relative geodesic $[\gamma,\gamma']$ from $\gamma$ to $\gamma'$ has to pass within a bounded distance of $\gamma_0\sigma$ and then $\gamma_0\sigma'$.
The projection of $\gamma'$ on $[e,\gamma]$ is on $\gamma_0$.
Denote by $\sigma''$ the first jump of $[e,\gamma]$ after $\gamma_0$.
The second part of Lemma~\ref{lemmarelativetripod} shows that $[\gamma,\gamma']$ passes within a bounded distance of $\gamma_0\sigma''$.
The only problem is that in our above construction, we changed the relative geodesic from $e$ to $\gamma$ translating the relative geodesic from $e$ to $[\gamma_0^{-1}\gamma]$.
Denote the new relative geodesic by $[[e,\gamma]]$
and denote by $\gamma_1$ the projection of $\gamma'$ on $[[e,\gamma]]$.
Then, the hyperbolicity of $\hat{\Gamma}$ together with the BCP property show that $d(\gamma_0,\gamma_1)\leq a_0$ for some $a_0$ independent of $\gamma$ and $\gamma'$.

Assume first that $d(e,\sigma)\geq a_0+1$.
Then, $\gamma_1$ has to be between $e$ and $\gamma_0$ on $[[e,\gamma]]$.
Since we did not change the part between $e$ and $\gamma_0$, we deduce that $\gamma_1$ also is between $\gamma_0$ and $\gamma$.
This proves that in this situation, $\gamma_1=\gamma_0$.
Similarly, if $d(e,\sigma')\geq a_1$ for some uniform $a_1$, then the BCP property shows that $\gamma_1=\gamma_0$.
In both cases, the second part of Lemma~\ref{lemmarelativetripod} indeed shows that $[\gamma,\gamma']$ passes within a bounded distance of $\gamma_0\sigma$ and then $\gamma_0\sigma'$.

We are left with the case where $d(e,\sigma)$ and $d(e,\sigma')$ are uniformly bounded.
In this case, the BCP property shows that $d(e,\sigma'')$ is also uniformly bounded,
that is,  $d(\gamma_0,\gamma_0\sigma'')$ is uniformly bounded.
Hence, $[\gamma,\gamma']$ passes within a bounded distance of $\gamma_0$, so within a bounded distance of $\gamma_0\sigma$ and $\gamma_0\sigma'$.

\medskip
Using the weak relative Ancona inequalities, we get
\begin{equation}\label{proofroughequadiff1}
\begin{split}
    &\sum_{\gamma\in \Gamma^{\gamma_0}}\sum_{\gamma'\in \Gamma_{\gamma_0}(\gamma)}G(\gamma_0,\gamma|r)G(\gamma,\gamma'|r)G(\gamma',\gamma_0|r)\\
    \asymp &\sum_{k,j}\sum_{\gamma\in \Gamma^{\gamma_0}_k}\sum_{\gamma'\in \Gamma_{\gamma_0}^j(\gamma)}G(e,\sigma|r)G(\gamma_0\sigma,\gamma|r)G(\gamma,\gamma_0\sigma|r)\\
    &\hspace{4cm}G(\sigma,\sigma'|r)G(\gamma_0\sigma',\gamma'|r)G(\gamma',\gamma_0\sigma'|r)G(\sigma',e|r).    
\end{split}
\end{equation}

We first fix $k$ and consider only the sum over $j\neq k$ in~(\ref{proofroughequadiff1}).
In this case, a relative geodesic from $\sigma$ to $\sigma'$ passes within a bounded distance of $e$.
This is obviously true if $d(e,\sigma)$ or $d(e,\sigma')$ is bounded and follows from the BCP property if both $d(e,\sigma)$ and $d(e,\sigma')$ are large enough.
We get that 
\begin{align*}
    &\sum_{j\neq k}\sum_{\gamma\in \Gamma^{\gamma_0}_k}\sum_{\gamma'\in \Gamma_{\gamma_0}^j(\gamma)}G(e,\sigma|r)G(\gamma_0\sigma,\gamma|r)G(\gamma,\gamma_0\sigma|r)\\
    &\hspace{4cm}G(\sigma,\sigma'|r)G(\gamma_0\sigma',\gamma'|r)G(\gamma',\gamma_0\sigma|r)G(\sigma',e|r)\\
    &\asymp \sum_{j\neq k}\sum_{\gamma\in \Gamma^{\gamma_0}_k}\sum_{\gamma'\in \Gamma_{\gamma_0}^j(\gamma)}G(\gamma_0,\gamma|r)G(\gamma,\gamma_0|r)G(\gamma_0,\gamma'|r)G(\gamma',\gamma_0|r).
\end{align*}
Translating everything by $\gamma_0^{-1}$, we bound this last term by
$$\sum_{j\neq k}\sum_{\gamma\in \Gamma^{e}_k}\sum_{\gamma'\in \Gamma_{e}^j(\gamma)}G(e,\gamma|r)G(\gamma,e|r)G(e,\gamma'|r)G(\gamma',e|r).$$
Indeed, by definition of $\Gamma^{\gamma_0}_k$, if $\gamma\in \Gamma^{\gamma_0}_k$, then $\gamma_0^{-1}\gamma\in \Gamma^{e}_k$.
Also, if the projection of $\gamma'$ on $[e,\gamma]$ is at $\gamma_0$, then the projection of $\gamma_0^{-1}\gamma'$ on $[\gamma_0^{-1},\gamma_0^{-1}\gamma]$ is at $e$.
Thus, the projection of $\gamma_0^{-1}\gamma'$ on $[e,\gamma_0^{-1}\gamma]$ also is at $e$, so that $\gamma_0^{-1}\gamma'\in \Gamma_{e}^j(\gamma)$.

Fix $\gamma\in \Gamma^{e}_k$.
If $\gamma'\in \Gamma_e^j(\gamma)$, in particular, the geodesic $[e,\gamma']$ starts with a jump in $\mathcal{H}_j$ so that $\gamma'\in \Gamma^e_j$.
We can thus bound the last sum by
$$\sum_{j\neq k}\sum_{\gamma\in \Gamma^{e}_k}\sum_{\gamma'\in\Gamma^{e}_j}G(e,\gamma|r)G(\gamma,e|r)G(e,\gamma'|r)G(\gamma',e|r)$$
which is itself bounded by
$$I^{(1)}(r)\sum_{\gamma\in \Gamma^{e}_k}G(e,\gamma|r)G(\gamma,e|r).$$
Summing over $k$, we finally have
$$\sum_k\sum_{j\neq k}\sum_{\gamma\in \Gamma^{\gamma_0}_k}\sum_{\gamma'\in \Gamma_{\gamma_0}^j(\gamma)}G(\gamma_0,\gamma|r)G(\gamma,\gamma_0|r)G(\gamma_0,\gamma'|r)G(\gamma',\gamma_0|r)\lesssim I^{(1)}(r)^2.$$
%Moreover, consider
%$$\sum_{\gamma''\in \Gamma^e_k}G(e,\gamma''|r)G(\gamma'',e|r).$$
%If the first jump in $[e,\gamma'']$ is in $\mathcal{H}_k$, the second one necessarily is in some $\mathcal{H}_j$, $j\neq k$.
%Weak Ancona inequalities then show that
%$$\sum_{\gamma''\in \Gamma^e_k}G(e,\gamma''|r)G(\gamma'',e|r)\lesssim \sum_{\sigma \in \mathcal{H}_k}\sum_{j\neq k}\sum_{\tilde{\gamma}\in \Gamma^e_i}G(e,\sigma|r)G(e,\tilde{\gamma}|r)G(\tilde{\gamma},e|r)G(\sigma,e|r).$$
%Since $\sum_{\sigma \in \mathcal{H}_k}G(e,\sigma|r)G(\sigma,e|r)$ is uniformly bounded according to Lemma~\ref{lemmafirstderivative},
%this proves that
%\begin{equation}\label{proofroughequadiff2}
%\sum_{j\neq k}\sum_{\tilde{\gamma}\in \Gamma^{e}_j}G(e,\gamma'|r)G(\gamma',e|r)\asymp I_1(r).
%\end{equation}
%We deduce from this that
%\begin{align*}
%    &\sum_{j\neq k}\sum_{\gamma\in \Gamma^{\gamma_0}_k}\sum_{\gamma'\in \Gamma_{\gamma_0}^j(\gamma)}G(\gamma_0,\gamma|r)G(\gamma,\gamma_0|r)G(\gamma_0,\gamma'|r)G(\gamma',\gamma_0|r)\\&\asymp I_1(r)\sum_{\gamma\in \Gamma^{e}_k}G(e,\gamma|r)G(\gamma,e|r).
%\end{align*}
%Summing over $k$, we thus  have that
%$$\sum_k\sum_{j\neq k}\sum_{\gamma\in \Gamma^{\gamma_0}_k}\sum_{\gamma'\in \Gamma_{\gamma_0}^j(\gamma)}G(\gamma_0,\gamma|r)G(\gamma,\gamma_0|r)G(\gamma_0,\gamma'|r)G(\gamma',\gamma_0|r)\asymp I_1(r)^2.$$

\medskip
To prove~(\ref{proofroughequadiff0}), we can thus only focus on the case where $j=k$ in~(\ref{proofroughequadiff1}).
We get
\begin{align*}
    &\sum_k\sum_{\gamma\in \Gamma^{\gamma_0}_k}\sum_{\gamma'\in \Gamma_{\gamma_0}^k(\gamma)}G(e,\sigma|r)G(\gamma_0\sigma,\gamma|r)G(\gamma,\gamma_0\sigma|r)\\
    &\hspace{4cm}G(\sigma,\sigma'|r)G(\gamma_0\sigma',\gamma'|r)G(\gamma',\gamma_0\sigma|r)G(\sigma',e|r).\\
\end{align*}
Translating $\gamma$ by $\sigma^{-1}\gamma_0^{-1}$ and $\gamma'$ by $(\sigma')^{-1}\gamma_0^{-1}$ on the left, we bound this sum by
\begin{align*}
\sum_{k}\sum_{\sigma,\sigma'\in \mathcal{H}_k}\sum_{\tilde{\gamma}}\sum_{\tilde{\gamma}'}G(e,\sigma|r)G(\sigma,\sigma'|r)G(\sigma',e|r)G(e,&\tilde{\gamma}|r)G(\tilde{\gamma},e|r)\\
&G(e,\tilde{\gamma}'|r)G(\tilde{\gamma}',e|r).
\end{align*}
We thus have
\begin{align*}
    &\sum_k\sum_{\gamma\in \Gamma^{\gamma_0}_k}\sum_{\gamma'\in \Gamma_{\gamma_0}^k(\gamma)}G(e,\sigma|r)G(\gamma_0\sigma,\gamma|r)G(\gamma,\gamma_0\sigma|r)\\
    &\hspace{4cm}G(\sigma,\sigma'|r)G(\gamma_0\sigma',\gamma'|r)G(\gamma',\gamma_0\sigma|r)G(\sigma',e|r)\\
    &\lesssim I^{(1)}(r)^2\sum_{k}I_k^{(2)}(r)
\end{align*}
which proves~(\ref{proofroughequadiff0}), which in turn, combined with~(\ref{proofroughequadiff-1}), proves~(\ref{roughequadiffupperbound}).

We now prove the lower bound, that is, we show that
\begin{equation}\label{roughequadifflowerbound}
     I^{(1)}(r)^3\left (1+ \sum_k I_k^{(2)}(r)\right )\lesssim I^{(2)}(r).
\end{equation}
If there is no parabolic subgroup or if they are all finite, then $\Gamma$ is hyperbolic and the result is given by \cite[Proposition~3.2]{Gouezel1}.
So we can assume that there is at least one parabolic subgroup and that it is infinite.
Since $I_k^{(2)}$ is bounded from below, it is sufficient to show that for every $1\leq k \leq N$,
$I^{(1)}(r)^3I_k^{(2)}(r)\lesssim I^{(2)}(r)$.
We fix such a $k$.
By definition,
$$I_k^{(2)}(r)=\sum_{\sigma,\sigma'\in \mathcal{H}_k}G(e,\sigma|r)G(\sigma,\sigma'|r)G(\sigma',e|r).$$
Up to a bounded multiplicative error, we can replace the sum over $\sigma$ and $\sigma'$ by a sum over $\sigma$ and $\sigma'$ such that $d(e,\sigma)\geq c$ and $d(e,\sigma')\geq c$, for some fixed $c$ that will be chosen later.
Also, according to Corollary~\ref{derivativeparabolicGreenfinite}, $\sum_{\sigma''\in \mathcal{H}_k}G(e,\sigma''|r)G(\sigma'',e|r)$ is finite, so that we can assume that $d(\sigma,\sigma')\geq c$ in the above sum.
We fix some large loxodromic element $\gamma_0\in\Gamma$ such that its projection on $\mathcal{H}_k$ in $\hat{\Gamma}$ is within a bounded $d$-distance of $e$.
According to Lemma~\ref{lemmaYangGouezel}, for every $\gamma\in \Gamma$, there exists $\gamma_1\in B(e,C)$ such that a relative geodesic from $e$ to $\tilde{\gamma}=\gamma_0\gamma_1\gamma$ passes within a bounded distance of $\gamma_0$.
In particular, the projection of $\tilde{\gamma}$ on $\mathcal{H}_k$ is within a bounded $d$-distance of $e$.

Denote by $\mathcal{O}(e)$ the set of $\gamma$ such that the projection of $\gamma$ on $\mathcal{H}_k$ is at $d$-distance at most $c/3$ of $e$.
We can reformulate the above discussion as follows, provided $c$ is large enough.
For every $\gamma \in \Gamma$, there exists $\gamma_2$ such that $d(e,\gamma_2)$ is uniformly bounded and $\tilde{\gamma}_e=\gamma_2\gamma$ is in $\mathcal{O}(e)$.
Also, for $\sigma\in \mathcal{H}_k$, denote by $\mathcal{O}(\sigma)$ the set of $\gamma$ such that its projection on $\mathcal{H}_k$ is at $d$-distance at most $c/3$ from $\sigma$.
Lemma~\ref{lemmaYangGouezel} shows that for any $\gamma$, there exists $\gamma_2$ such that $d(e,\gamma_2)$ is uniformly bounded and any relative geodesic from $e$ to $\tilde{\gamma}_{\sigma}=\sigma\gamma_2\gamma$ passes within a bounded distance of $\sigma$.
If $c$ is large enough, then $\tilde{\gamma}_{\sigma}\in \mathcal{O}(\sigma)$.
Recall that $d(e,\sigma)\geq c$, $d(e,\sigma')\geq c$ and $d(\sigma,\sigma')\geq c$ in the above sum.
According to the BCP property, if $c$ is large enough, we also have that
\begin{enumerate}
    \item the sets $\mathcal{O}(\sigma)$, $\mathcal{O}(\sigma')$ and $\mathcal{O}(e)$ are all disjoint,
    \item if $\gamma \in \mathcal{O}(\sigma)$ and $\gamma'\in \mathcal{O}(\sigma')$, then any relative geodesic from $\gamma$ to $\gamma'$ passes first within a bounded distance of $\sigma$, then within a bounded distance of $\sigma'$, and similarly for $\mathcal{O}(e)$ and $\mathcal{O}(\sigma)$ and for $\mathcal{O}(\sigma')$ and $\mathcal{O}(e)$.
\end{enumerate}

Let $\gamma\in \Gamma$, then we have
$$G(e,\gamma|r)G(\gamma,e|r)\lesssim G(e,\tilde{\gamma}_{e}|r)G(\tilde{\gamma}_e,e|r).$$
Also, for any $\sigma\in \mathcal{H}_k$, we have
$$G(e,\gamma|r)G(\gamma,e|r)\lesssim G(\sigma,\tilde{\gamma}_{\sigma}|r)G(\tilde{\gamma}_{\sigma},\sigma|r).$$
We thus get, using weak relative Ancona inequalities,
\begin{align*}
    I^{(1)}(r)^3I_k^{(2)}(r)&\lesssim \sum_{\sigma,\sigma'}G(e,\sigma|r)\left (\sum_{\gamma}G(\sigma,\tilde{\gamma}_{\sigma}|r)G(\tilde{\gamma}_{\sigma},\sigma|r)\right )\\
    &\hspace{1cm}G(\sigma,\sigma'|r)\left (\sum_{\gamma'}G(\sigma',\tilde{\gamma}'_{\sigma'}|r)G(\tilde{\gamma}'_{\sigma'},\sigma'|r)\right )\\
    &\hspace{1cm}G(\sigma',e|r)\left (\sum_{\gamma''}G(e,\tilde{\gamma}''_{e}|r)G(\tilde{\gamma}''_{e},e|r)\right ).\\
    &\lesssim \sum_{\sigma,\sigma'}\sum_{\gamma,\gamma',\gamma''}G(\tilde{\gamma}''_{e},\tilde{\gamma}_{\sigma}|r)G(\tilde{\gamma}_{\sigma},\tilde{\gamma}'_{\sigma'}|r)G(\tilde{\gamma}'_{\sigma'},\tilde{\gamma}''_{e}|r).
\end{align*}
The sets $\mathcal{O}(e)$, $\mathcal{O}(\sigma)$ and $\mathcal{O}(\sigma')$ are all disjoint.
For $\tilde{\gamma} \in \mathcal{O}(\sigma)$, $\tilde{\gamma}'\in \mathcal{O}(\sigma')$ and $\tilde{\gamma}''\in \mathcal{O}(e)$, let $\overline{\gamma}=(\tilde{\gamma}'')^{-1}\tilde{\gamma}$ and $\overline{\gamma}'=(\tilde{\gamma}'')^{-1}\tilde{\gamma}'$.
Then, the projection of $\overline{\gamma}'$ in $\hat{\Gamma}$ on a relative geodesic $[e,\overline{\gamma}]$ is within a bounded $d$-distance of $(\tilde{\gamma}'')^{-1}$.
In other words, up to a bounded error, $\tilde{\gamma}, \tilde{\gamma}',\tilde{\gamma}''$ determine $\overline{\gamma}$ and $\overline{\gamma}'$.
Moreover, for fixed $\gamma$, there is a uniformly finite number of $\sigma$ such that $\gamma \in \mathcal{O}(\sigma)$, according to the BCP property.
This proves that
$$I^{(1)}(r)^3I_k^{(2)}(r)\lesssim \sum_{\overline{\gamma},\overline{\gamma}'}G(e,\overline{\gamma}|r)G(\overline{\gamma},\overline{\gamma}'|r)G(\overline{\gamma}',e|r)=I^{(2)}(r).$$
This proves~(\ref{roughequadifflowerbound}), which concludes the proof.
\end{proof}

We thus proved Theorem~\ref{maintheoremGreen}.
More generally, we have the following combinatorial result, which relates $I^{(j)}(r)$ with $I_k^{(j)}(r)$.

\begin{lemma}\label{roughequadiffhighorder}
There exists a constant $C> 1$ such that for every $r\in [0,R_{\mu}]$, for every $j\geq 2$,
\begin{align*}
    &\frac{I^{(j)}(r)}{I^{(1)}(r)}\leq C \sum_{l\geq 2}C^{l}\sum_{i_1+...+i_l=j}I^{(i_1)}(r)...I^{(i_l)}(r) \left (1+\sum_{1\leq k\leq N}I_{k}^{(l)}(r)\right )+\\
    &\sum_{m\geq 2}\sum_{j_1+...j_m=j}C^m\prod_{p=1}^m\left (\sum_{l\geq 1}C^{l}\sum_{i_1+...+i_l=j_p}I^{(i_1)}(r)...I^{(i_l)}(r) \left (1+\sum_{1\leq k\leq N}I_{k}^{(l)}(r)\right )\right ).
\end{align*}
\end{lemma}

\begin{proof}
The proof goes the same way as above, but instead of summing over $\gamma,\gamma'$, we have to sum over $\gamma_1,\gamma_2,...,\gamma_j$ so we have to be more precise.
We choose for every $\gamma$ a relative geodesic $[e,\gamma]$ from $e$ to $\gamma$.
We denote by $\gamma_0$ the last element of $[e,\gamma_1]$ such that relative geodesics from $e$ to $\gamma_1,...,\gamma_j$ pass a $d$-distance at most $c$ from $\gamma_0$, for some $c$ that will be chosen later.

First, we consider the sub-sum $\Sigma_1$ over points $\gamma_1,...,\gamma_j$ such that the first jump $\sigma_i$ of $[\gamma_0,\gamma_i]$ is in the same parabolic subgroup $\mathcal{H}_k$, for every $1\leq i\leq j$.
Fix a constant $c$ and denote by $i_1$ the last index such that $d(\sigma_i,\sigma_1)\leq c$.
As in the proof of Lemma~\ref{roughequadiff}, if $c$ is large enough, then a relative geodesic from $\gamma_{i_1}$ to $\gamma_{i_1+1}$ has to pass within a bounded distance of $\gamma_0\sigma_i$.
Weak relative Ancona inequalities then give
$$G(\gamma_{i_1},\gamma_{i_1+1}|r)\lesssim G(\gamma_{i_1},\gamma_0\sigma_{1}|r)G(\sigma_{1},\sigma_{i_1+1}|r)G(\gamma_0\sigma_{i_1+1},\gamma_{i_1+1}|r).$$
Denote then $i_2$ the last index such that $d(\sigma_{i_1+i_2},\sigma_{i_1+1})\leq c$.
Similarly,
$$G(\gamma_{i_2},\gamma_{i_2+1}|r)\lesssim G(\gamma_{i_2},\gamma_0\sigma_{i_1+1}|r)G(\sigma_{i_1+1},\sigma_{i_2+1}|r)G(\gamma_0\sigma_{i_2+1},\gamma_{i_2+1}|r).$$
We go on and get a decomposition of $j$ as $i_1+i_2+...+i_l$.
Combining all the inequalities above, we have
\begin{align*}
    &G(\gamma_0,\gamma_1|r)G(\gamma_1,\gamma_2|r)...G(\gamma_{j-1},\gamma_j|r)G(\gamma_j,\gamma_0|r)\leq \\
    &C^{l}G(e,\sigma_1|r)G(\sigma_1,\sigma_{(i_1+1)}|r)...G(\sigma_{(i_1+...+i_{l-1}+1)},e|r)\\
    &G(\gamma_0\sigma_1,\gamma_1|r)G(\gamma_1,\gamma_2)...G(\gamma_{i_1-1},\gamma_{i_1}|r)G(\gamma_{i_1},\gamma_0\sigma_1|r)\\
    &G(\gamma_0\sigma_{i_1+1},\gamma_{i_1+1}|r)G(\gamma_{i_1+1},\gamma_{i_1+2}|r)...G(\gamma_{i_1+i_2-1},\gamma_{i_1+i_2}|r)G(\gamma_{i_1+i_2},\gamma_0\sigma_{i_1+1}|r)\\
    &\hspace{6cm}...\\
    &G(\gamma_0\sigma_{(i_1+...+i_{l-1}+1)},\gamma_{i_1+...+i_{l-1}+1}|r)G(\gamma_{i_1+...+i_{l-1}+1},\gamma_{i_1+...+i_{l-1}+2}|r)...\\
    &\hspace{5cm}...G(\gamma_{i-1},\gamma_j|r)G(\gamma_j,\gamma_0\sigma_{(i_1+...+i_{l-1}+1)}|r),
\end{align*}
for some constant $C\geq 0$.
By definition of $\gamma_0$, we necessarily have $l\geq 2$.

We obtain $\Sigma_1$ by summing over every such decomposition of $j$ and over such $\sigma_1$,...,$\sigma_{(i_1+...+i_{l-1}+1)}$.
The second line in the right-hand side will give a contribution bounded by $I^{(i_1)}(r)$, the third line by $I^{(i_2)}(r)$ and so on.
The first line will give a contribution bounded by $I_k^{(l)}(r)$, so this contribution will be bounded by
$\sum_kI_k^{(l)}(r)$.
We thus get
$$\Sigma_1\leq \sum_{l\geq 2}\sum_{i_1+...+i_l=j}C^{l}I^{(i_1)}(r)...I^{(i_l)}(r)\left (1+\sum_{1\leq k\leq N}I_k^{(l)}(r)\right).$$

Assume now that the jumps $\sigma_1,...,\sigma_{j_1}$ are in the same parabolic subgroup $\mathcal{H}_{k_1}$ and that the other jumps $\sigma_{j_1+1},...,\sigma_{j}$ lie in some other parabolic subgroup $\mathcal{H}_{k_2}$.
Then,
$$G(\gamma_{j_1},\gamma_{j_1+1}|r)\leq CG(\gamma_{j_1},\gamma_0|r)G(\gamma_0,\gamma_{j_1+1}|r).$$
We then get an upper bound for
$$G(\gamma_0,\gamma_1|r)...G(\gamma_{j_1-1},\gamma_{j_1}|r)G(\gamma_{j_1},\gamma_0|r)$$
and
$$G(\gamma_0,\gamma_{j_1+1}|r)...G(\gamma_{j-1},\gamma_{j}|r)G(\gamma_{j},\gamma_0|r)$$
as above, except that we do not necessarily have $l\geq 2$.
Summing over all such possible decomposition, we finally get a sum $\Sigma_2$ which satisfies
\begin{align*}
    \Sigma_2\leq C&\left (\sum_{l\geq 1}C^{l+1}\sum_{i_1+...+i_l=j_1}I^{(i_1)}(r)...I^{(i_l)}(r)\left (1+\sum_{1\leq k\leq N}I_{k}^{(l)}(r)\right )\right )\\
&\left (\sum_{l\geq 1}C^{l+1}\sum_{i_1+...+i_l=j-j_1}I^{(i_1)}(r)...I^{(i_l)}(r)\left (1+\sum_{1\leq k\leq N}I_{k}^{(l)}(r)\right )\right ).
\end{align*}

To conclude, we decompose in general $j$ as $j_1+...+j_m$, where $\sigma_1,...,\sigma_{j_1}$ lie in the same parabolic $\mathcal{H}_{k_1}$, $\sigma_{j_1+1},...,\sigma_{j_1+j_2}$ lie in the same $\mathcal{H}_{k_2}$ and so on.
We similarly get sums $\Sigma_m$ which satisfy
\begin{align*}
    \Sigma_m\leq &\sum_{j_1+...j_m=j}C^{2m}\prod_{p=1}^m\left (\sum_{l\geq 1}C^{3l}\sum_{i_1+...+i_l=j_p}I^{(i_1)}(r)...I^{(i_l)}(r)\right .\\
    &\hspace{6cm}\left .\left (1+\sum_{1\leq k\leq N}I_{k}^{(l)}(r)\right )\right ).
\end{align*}
This proves the lemma, summing over all possible $m$.
\end{proof}

We can use this upper bound to prove the second main result of this subsection, namely that spectral degenerescence of the measure $\mu$ implies that $\mu$ is divergent.

\begin{proposition}\label{relationR_kG'(R)}
If for parabolic subgroup $\mathcal{H}_k\in \Omega_0$, $1\leq k\leq N$, we have $R_k>1$, then $\frac{d}{dr}_{|r=R_\mu}G(e,e|r)=+\infty$.
\end{proposition}

\begin{proof}
If for every $k$, $R_k>1$, then Proposition~\ref{propderivativesGreensubexponential} shows that for every $j\geq 2$ and for every $k$, $I_{k}^{(j)}(R_{\mu})\leq c_0C_0^j$ for some $c_0\geq0$ and $C_0>0$.
Lemma~\ref{roughequadiffhighorder} shows that for some $c\geq 0$ and $C>0$,
\begin{equation}\label{equationI_j}
\begin{split}
    &\frac{I^{(j)}(R_{\mu})}{I^{(1)}(R_{\mu})}\leq c\sum_{l\geq 2}C^{l}\sum_{i_1+...+i_l=j}I^{(i_1)}(R_{\mu})...I^{(i_l)}(R_{\mu})\\
    &\hspace{.5cm}+c\sum_{m\geq 2}\sum_{j_1+...j_m=j}C^m\prod_{p=1}^m\left (\sum_{l\geq 1}C^{l}\sum_{i_1+...+i_l=j_p}I^{(i_1)}(R_{\mu})...I^{(i_l)}(R_{\mu})\right ).
\end{split}
\end{equation}
Notice that we can enlarge $C$ if necessary.
We will prove by contradiction that $\frac{d}{dr}_{|r=R_\mu}G(e,e|r)=+\infty$.
Our goal is to prove that $I^{(j)}(R_{\mu})$ grows at most exponentially.
Define inductively $J_1=cI^{(1)}(R_{\mu})$ and for $j\geq 2$,
\begin{align*}
    J_j=&cI^{(1)}(R_{\mu})\sum_{l\geq 2}C^{l}\sum_{i_1+...+i_l=j}J_{i_1}...J_{i_l}\\
    &\hspace{.5cm}+cI^{(1)}(R_{\mu})\sum_{m\geq 2}\sum_{j_1+...j_m=j}C^m\prod_{p=1}^m\left (\sum_{l\geq 1}C^{l}\sum_{i_1+...+i_l=j_p}J^{(i_1)}...J^{(i_l)}\right ).
\end{align*}

First, since $\frac{d}{dr}_{|r=R_\mu}G(e,e|r)<+\infty$, Lemma~\ref{lemmafirstderivative} shows that $I^{(1)}(R_{\mu})$ is finite, hence $J_1$ also is finite.
Thus,~(\ref{equationI_j}) combined with an induction argument shows that
\begin{equation}\label{I_jsmallerJ_j}
    I^{(j)}(R_{\mu})\leq J_j.
\end{equation}

Our goal is now to prove that $J_j$ grows at most exponentially in $j$.
Define the power series
$$J(z)=\sum_{j\geq 1}J_jz^j.$$
Assume first that its radius of convergence is positive.
Using the definition of $J_j$, we get
\begin{align*}
    \frac{J(z)}{J_1}&=z+\sum_{j\geq 2}\left (\sum_{l\geq 2}C^{l}\sum_{j_1+...+j_l=j}J_{j_1}...J_{j_l}\right )z^j\\
    &\hspace{1cm}+\sum_{j\geq 2}\sum_{m\geq 2}\sum_{j_1+...j_m=j}C^m\prod_{p=1}^m\left (\sum_{l\geq 1}C^{l}\sum_{i_1+...+i_l=j_p}J_{i_1}...J_{i_l}\right )z^j\\
    &=z+\sum_{l\geq 2}\left (\sum_{j\geq 2}\sum_{j_1+...+j_l=j}\left (J_{j_1}z^{j_1}\right )...\left (J_{j_l}z^{j_l}\right )\right )C^l\\
    &\hspace{1cm}+\sum_{m\geq 2}\sum_{j\geq 2}\sum_{j_1+...j_m=j}C^m\prod_{p=1}^m\left (\sum_{l\geq 1}C^{l}\sum_{i_1+...+i_l=j_p}J_{i_1}...J_{i_l}\right )z^j.
\end{align*}
The Cauchy product formula shows that for $l\geq 2$,
$$\sum_{j\geq 2}\sum_{j_1+...+j_l=j}\left (J_{j_1}z^{j_1}\right )...\left (J_{j_l}z^{j_l}\right )=J(z)^l.$$
We thus get
\begin{align*}
\frac{J(z)}{J_1}=&z+\sum_{l\geq 2}C^lJ(z)^l\\
&+\sum_{m\geq 2}\sum_{j\geq 2}\sum_{j_1+...j_m=j}C^m\prod_{p=1}^m\left (\sum_{l\geq 1}C^{l}\sum_{i_1+...+i_l=j_p}J_{i_1}...J_{i_l}\right )z^j.
\end{align*}
Let 
$$K_q=\sum_{l\geq 1}C^l\sum_{i_1+...+i_l=q}J_{i_1}...J_{i_l}$$
and let
$K(z)=\sum_{q\geq 1}K_qz^q$.
Then, similarly,
$$\sum_{m\geq 2}\sum_{j\geq 2}\sum_{j_1+...j_m=j}C^m\prod_{p=1}^m\left (\sum_{l\geq 1}C^{l}\sum_{i_1+...+i_l=j_p}J_{i_1}...J_{i_l}\right )z^j=\sum_{m\geq2}C^mK(z)^m$$
We get
$$\frac{J(z)}{J_1}=z+\sum_{l\geq 2}C^lJ(z)^l+\sum_{m\geq2}C^mK(z)^m=z+\frac{C^2J(z)^2}{1-CJ(z)}+\frac{C^2K(z)^2}{1-CK(z)}.$$
We also have
$$K(z)=\sum_{l\geq 1}C^lJ(z)^l=\frac{CJ(z)}{1-CJ(z)},$$
so that we finally get
\begin{equation}\label{polynomialequation}
    aJ(z)^3-b(z)J(z)^2+c(z)J(z)-d(z)=0.
\end{equation}
where
\begin{align*}
    a&=C(C+C^2)\left (\frac{1}{J_1}+C\right ),\\
    b(z)&=\frac{2C+C^2}{J_1}+C^2+C^4+(C^2+C^3)z,\\
    c(z)&=\frac{1}{J_1}+(2C+C^2)z,\\
    d(z)&=z.
\end{align*}
As noticed above, we can enlarge $C$.
For $z=0$, the polynomial equation
\begin{equation}\label{polynomialequationsolution}
    ax^3-b(0)x^2+c(0)x-d(0)=0
\end{equation}
has three solutions, namely $x=0$ and
$$x=\frac{J_1C^4+C^2(1+J_1)+2C\pm C^2\sqrt{C^4J_1^2+2 C^2 J_1 (1 + J_1)+4C J_1+ (J_1-1)^2}}{2C(C+C^2)(1+J_1C)}.$$
As $C$ tends to infinity, the second solution tends to 1
and the third one is asymptotic to $\frac{-4}{2C(C+C^2)(1+J_1C)}<0$.
Thus, if $C$ is large enough, the three solutions are distinct.
We fix such a $C$.
The implicit function Theorem shows that there is an analytic function $\tilde{J}(z)$ of $z$ in a neighborhood of 0 which is a solution of~(\ref{polynomialequationsolution}) and satisfies $\tilde{J}(0)=0$.
This proves that the radius of convergence of $\tilde{J}(z)$ is positive.
Moreover, the coefficients of the power series defined by $J(z)$ and $\tilde{J}(z)$ satisfy the same induction equation, so that they coincide and the radius of convergence of $J(z)$ is positive.

In particular, there exists
$d_0\geq 0$, $D_0>0$ such that
$J_j\leq d_0D_0^j$, so that according to~(\ref{I_jsmallerJ_j}),
$I^{(j)}(R_{\mu})\leq dD^jI^{(1)}(R_{\mu})$.
Proposition~\ref{propderivativesGreensubexponential} shows that $\frac{G^{(j)}_r}{k!}\leq dD^j$ at $r=R_{\mu}$, for some $d\geq0$ and $D>0$, which is a contradiction.
Thus, $I^{(1)}=+\infty$, that is, $G^{(1)}_{R_{\mu}}=+\infty$.
This concludes the proof.
\end{proof}

\section{From the rough estimates of the Green function to the rough local limit theorem}\label{SectionTauberian}
Our goal in this section is to prove Theorem~\ref{maintheorem}.
We start with the following result, which was proved by Mateusz Kwaśnicki in \cite{354036}.
We reproduce the proof for simplicity.

\begin{theorem}\label{thmweakTauberian}
Let $A(z)=\sum a_nz^n$ be a power series with non-negative coefficients $a_n$ and radius of convergence 1.
Let $\beta>0$.
Then, $\sum_{n\geq 0}a_ns^n  \asymp 1/(1-s)^\beta$ for $s\in [0,1)$ if and only if $\sum_{k=0}^na_k\asymp n^{\beta}$.
The implicit constants are asked not to depend on $s$ and $n$ respectively.
\end{theorem}

\begin{proof}
Following \cite{BinghamGoldieTeugels}, we denote by $OR$ the set of positive functions $f$ such that for every $\lambda>0$,
$$\limsup_{x\to \infty}\frac{f(\lambda x)}{f(x)}<\infty.$$
Then, \cite[Theorem~2.10.2]{BinghamGoldieTeugels} shows that for every non-decreasing measurable function $U$ with positive $\liminf$ and vanishing on $(-\infty,0)$, the following are equivalent
\begin{enumerate}
    \item $U\in OR$,
    \item $t\mapsto \hat{U}(1/t)\in OR$,
    \item $U(t)\asymp \hat{U}(1/t), t>0$.
\end{enumerate}
Here, $\hat{U}$ is the Laplace transform of the Stieltjes measure $U(dx)$.
Precisely,
$$\hat{U}(t)=\int_{0}^{\infty}\mathrm{e}^{-t x}U(dx).$$

We apply this to $U(x)=\sum_{k=0}^{\lfloor x\rfloor}a_k$.
We then have
$$\hat{U}(t)=\sum_{n\geq 0}a_n\mathrm{e}^{-tn}=A(\mathrm{e}^{-t}).$$
Assuming that $A(s)=\sum_{n\geq 0}a_ns^n \asymp 1/(1-s)^\beta$ for $s\in [0,1)$, we get
$$\hat{U}(1/t)=A(\mathrm{e}^{-1/t})\asymp (1-\mathrm{e}^{-1/t})^{-\beta},$$ so that $t\mapsto \hat{U}(1/t)\in OR$.
Thus, $U(t)\asymp \hat{U}(1/t)$ and so
$$\sum_{k=0}^na_k\asymp (1-\mathrm{e}^{-1/n})^{-\beta}\asymp n^{\beta}.$$

Conversely, assuming that $\sum_{k=0}^na_k\asymp n^{\beta}$, then $U\in OR$ and so we also have $U(t)\asymp \hat{U}(1/t)$.
Consequently,
$A(\mathrm{e}^{-1/t})\asymp t^\beta$ and so
$$\sum_{n\geq 0}a_ns^n  \asymp 1/(1-s)^\beta.$$
This concludes the proof.
\end{proof}

We want to apply Theorem~\ref{thmweakTauberian} to prove Theorem~\ref{maintheorem}.
Let $\Gamma$ be a non-elementary relatively hyperbolic group.
Let $\mu$ be a finitely supported, admissible and symmetric probability measure on $\Gamma$.
Assume that the corresponding random walk is aperiodic and spectrally positive-recurrent.
We write $G'(r)=\frac{d}{dr}G(e,e|r)$ and $G''(r)=\frac{d^2}{dr^2}G(e,e|r)$ for simplicity.
Theorem~\ref{maintheoremGreen} shows that
$$\frac{G''(r)}{(G'(r))^3}\asymp 1.$$
Letting $r\leq R<R_\mu$, integrating these two inequalities between $r$ and $R$ yields
$$\frac{1}{G'(r)^2}-\frac{1}{G'(R)^2}\asymp R-r.$$
By monotone convergence $G'(R)$ tends to $G'(R_\mu)$ as $R$ converges to $R_\mu$.
Since we are assuming that $\mu$ is divergent, we get
$$\frac{1}{G'(r)^2}\asymp R_\mu-r,$$
so that
$$G'(r)\asymp (R_\mu-r)^{1/2}.$$
Theorem~\ref{thmweakTauberian}, applied to $a_{n-1}=nR_\mu^np_n(e,e)$, shows that
$$\sum_{k=1}^{n+1}kR_\mu^kp_k(e,e)\asymp n^{1/2}.$$
Classically, when $b_k$ is a non-increasing sequence satisfying that $\sum_{k=1}^nb_k\asymp n^{\beta}$, one can prove that $b_n\asymp n^{\beta-1}$.
We will prove a similar statement below.
Unfortunately, there is no chance to guaranty that $nR_\mu^np_n(e,e)$ is non-increasing, so we cannot deduce yet that $R_\mu^{n}p_n(e,e)\asymp n^{-3/2}$.

However, \cite[Theorem~9.4]{GouezelLalley} shows that there exists $\alpha>0$ such that
\begin{equation}\label{equationdecreasing}
    nR_\mu^np_n(e,e)=nq_n + O\left (\mathrm{e}^{-\alpha n}\right ),
\end{equation}
where $q_n\geq 0$ is non-increasing.
In particular,
$$\sum_{k=1}^{n+1}kR_\mu^kp_k(e,e) = \sum_{k=1}^{n+1}kq_k + O(1)$$
and so
$$\sum_{k=1}^{n+1}kq_k\asymp n^{1/2}.$$

\begin{lemma}
Let $b_n$ be a non-increasing sequence of non-negative numbers and let $0<\beta<1$.
Assume that $\sum_{k=1}^nkb_k\asymp n^{\beta}$.
Then, $b_n\asymp n^{\beta-2}$.
\end{lemma}

\begin{proof}
First, we see that
$$\sum_{k=1}^nkb_k\geq b_n \sum_{k=1}^nk.$$
Now, $n^2\lesssim \sum_{k=1}^nk$, so
$$b_n\lesssim n^{-2}\sum_{k=1}^nkb_k\lesssim n^{-2}n^\beta,$$
which proves the first inequality.

We prove the second one.
Let $c$ and $C$ be such that
$$cn^{\beta}\leq \sum_{k=1}^nkb_k\leq C n^{\beta}.$$
We fix a large constant $A$ such that $2C\leq c A^{\beta}$.
Then $2Cn^\beta\leq c (An)^{\beta}$ and so
$$2Cn^{\beta}\leq \sum_{k=1}^{\lceil An\rceil}kb_k\leq \sum_{k=1}^nkb_k + b_n\sum_{n+1}^{\lceil An\rceil}k.$$
Note that $\sum_{n+1}^{\lceil An\rceil}k\lesssim n^2$, hence there exists $C'$ such that
$$2Cn^{\beta}\leq \sum_{k=1}^nkb_k + C'b_n n^2\leq Cn^{\beta} +C'b_n n^2.$$
We thus obtain
$Cn^{\beta}\leq C'b_n n^{2}$, which concludes the proof.
\end{proof}

Applied to our situation, this lemma shows that $q_n\asymp n^{-3/2}$.
We then use again~(\ref{equationdecreasing}) to deduce that $R_\mu^{n}p_n(e,e)\asymp n^{-3/2}$.
This proves Theorem~\ref{maintheorem}. \qed

\bibliographystyle{plain}
\bibliography{LLT_1}

\end{document}